\documentclass[11pt]{amsart} 
\usepackage{amscd,amssymb,amsxtra}
\usepackage{mathrsfs}  
\DeclareSymbolFontAlphabet{\mathrsfs}{rsfs}
\usepackage[mathscr]{eucal} 
\usepackage{mathabx}\DeclareMathSymbol{\emptyset}{\mathord}{symbols}{59}
                    \DeclareMathSymbol{\boxtimes}{\mathbin}{AMSa}{"02}
\usepackage{comment}
\usepackage{color}  
\newcommand\red{\color{red}}
\newcommand\blue{\color{blue}}
\usepackage{enumitem} 
\usepackage{float}
\usepackage{marginnote}

\makeatletter
\let\@secnumfont\bfseries
\def\section{\@startsection{section}{1}%
  \z@{4\linespacing\@plus\linespacing}{\linespacing}%
  {\bfseries\centering}}
\def\introsection{\@startsection{section}{1}%
  \z@{3\linespacing\@plus\linespacing}{\linespacing}%
  {\bfseries\centering}}
\def\subsection{\@startsection{subsection}{2}%
   \z@{1.25\linespacing\@plus.7\linespacing}{.5\linespacing}%
   {\normalfont\bfseries}}
\def\subsectionsinline{\def\subsection{\@startsection{subsection}{2}%
  \z@{1\linespacing\@plus.7\linespacing}{-.5em}%
  {\normalfont\bfseries}}}

\newenvironment{qedequation}{%
   \pushQED{\qed}%
   \incr@eqnum
   \mathdisplay@push
   \st@rredfalse \global\@eqnswtrue
   \mathdisplay{equation}%
}{%
   \endmathdisplay{equation}%
   \mathdisplay@pop
   \ignorespacesafterend
   \popQED\@endpefalse
}

\makeatother

\theoremstyle{definition}

\newtheorem*{definition*}{Definition}
\newtheorem*{example*}{Example}
\newtheorem*{problem*}{\color{blue}Problem}
\newtheorem*{exercise*}{Exercise}
\newtheorem*{question*}{\color{blue}Question}
\newtheorem*{project*}{\color{blue}Project}
\newtheorem*{construction*}{Construction}

\theoremstyle{remark}

\newtheorem{remark}[equation]{Remark}

\newtheorem*{note*}{Note}
\newtheorem*{notation*}{Notation}
\newtheorem*{remark*}{Remark}
\newtheorem*{data*}{Data}

\theoremstyle{plain}
\newtheorem{theorem}[equation]{Theorem}
\newtheorem{corollary}[equation]{Corollary}

\newtheorem{proposition}[equation]{Proposition}

\newtheorem*{theorem*}{Theorem}
\newtheorem*{corollary*}{Corollary}
\newtheorem*{lemma*}{Lemma}
\newtheorem*{proposition*}{Proposition}
\newtheorem*{conjecture*}{Conjecture}
\newtheorem*{claim*}{Claim}
\newtheorem*{proposal*}{Proposal}
\newtheorem*{conclusion*}{Conclusion}
\newtheorem*{hypothesis*}{Hypothesis}
\newtheorem*{assumption*}{Assumption}

\newenvironment{proof*}[1][\proofname]{
  \begin{proof}[#1]}{  
\end{proof}}

\numberwithin{equation}{section}

\definecolor{refkey}{rgb}{0,.6,.4}

\renewcommand{\:}{\colon}

\newcommand{\CC}{{\mathbb C}}
\newcommand{\CP}{{\mathbb C\mathbb P}}

\DeclareMathOperator{\id}{id}

\DeclareMathOperator{\Map}{Map}

\DeclareMathOperator{\pt}{pt}

\newcommand{\RR}{{\mathbb R}}
\newcommand{\TT}{\mathbb T}
\DeclareMathOperator{\Spin}{Spin}

\newcommand{\ZZ}{{\mathbb Z}}

\newcommand{\chiup}{\raise.5ex\hbox{$\chi$}}
\newcommand{\cir}{S^1}

\newcommand{\inv}{^{-1}}
\DeclareRobustCommand{\mstrut}{^{\vphantom{1*\prime y\vee M}}}

\newcommand{\res}[1]{\negmedspace\bigm|\mstrut_{#1}}

\newcommand{\temsquare}{\raise3.5pt\hbox{\boxed{ }}}

\newcommand{\zmod}[1]{\ZZ/#1\ZZ}

\newcommand{\zt}{\zmod2}

\DeclareMathOperator{\SO}{SO}

\let\O\relax
\DeclareMathOperator{\O}{O}

\DeclareMathOperator{\SL}{SL}

\usepackage[all,2cell,color]{xy}\renewcommand{\cir}{\ensuremath{S^1}}
\usepackage{graphicx}\usepackage{epsf} 
\usepackage[colorlinks,backref=page,citecolor=refkey]{hyperref}
\let\O\relax\DeclareMathOperator{\O}{O}
\definecolor{refkey}{rgb}{0,.8,.2}\definecolor{labelkey}{rgb}{1,0,0} 

\DeclareMathOperator{\Pin}{Pin}
\DeclareMathOperator{\Ad}{Ad}
\DeclareMathOperator{\Bord}{Bord}
\DeclareMathOperator{\Line}{Line}
\DeclareMathOperator{\Li}{Li}
\DeclareMathOperator{\Sect}{Sect}
\DeclareMathOperator{\curv}{curv}
\newcommand{\BNGG}{B^{}_{\nabla }G}
\newcommand{\BNG}{B^{}_{\nabla }\Cx}

\newcommand{\CZt}{\CC/\Zt}
\newcommand{\CZ}{\CC/\ZZ}
\newcommand{\Cx}{\CC^\times}
\newcommand{\ENGG}{E^{}_{\nabla }G}
\newcommand{\ENG}{E^{}_{\nabla }\Cx}
\newcommand{\F}[2]{\rF(#1;#2)}

\newcommand{\LC}{\Line\mstrut _\CC}
\newcommand{\MTp}{\sM _T'}
\newcommand{\MT}{\sM\mstrut _T}
\newcommand{\MXp}{\sM _X'}
\newcommand{\MX}{\sM\mstrut _X}
\newcommand{\OC}[2]{\Omega^{#1}_{#2}(\CC)}
\newcommand{\RZ}{\RR/\ZZ}
\newcommand{\Tu}{\Theta ^{\textnormal{univ}}}
\newcommand{\Zo}{\ZZ(1)}
\newcommand{\Zt}{\ZZ(2)}
\newcommand{\asp}[1]{\alpha_{\textnormal{spin}}(#1)}
\newcommand{\brdn}{\Bord_{\langle 2,3  \rangle}(\SO_3\times (\Cx)^\nabla )}
\newcommand{\brds}{\Bord _{\langle 2,3  \rangle}(\Spin_3\times (\Cx)^\nabla )}

\newcommand{\cco}{\widecheck{c}_1}
\newcommand{\ce}{\widecheck{\eta}}
\newcommand{\cla}{\widecheck{\lambda }}
\newcommand{\conn}[1]{\sA_{#1}}
\newcommand{\co}{\widecheck{\omega }}
\newcommand{\dEs}{\widecheck{E}\mstrut _{\CC}}
\newcommand{\dE}{\widecheck{E}_{\CC}}
\newcommand{\dH}{\widecheck{H}_{\CC}}
\newcommand{\da}{\dot{\alpha }}
\newcommand{\hGe}{\widehat{\Gamma }_\epsilon }
\newcommand{\hMTp}{\widehat{\sM} _T'}
\newcommand{\hMT}{\widehat{\sM}\mstrut _T}
\newcommand{\hMXp}{\widehat{\sM} _X'}
\newcommand{\hMX}{\widehat{\sM}\mstrut _X}
\newcommand{\har}{\widehat{r}}
\newcommand{\he}{\hat{\eta }}
\newcommand{\hr}{\hat{\rho }}
\newcommand{\hsL}{\widehat{\sL}}

\newcommand{\htao}{\widehat{\tau}_0}
\newcommand{\hto}{\hat{t}_0}
\newcommand{\rF}{\mathrsfs{F}}
\newcommand{\rS}{\mathrsfs{S}}
\newcommand{\sA}{\mathcal{A}}
\newcommand{\sF}{\mathscr{F}}
\newcommand{\sLC}{s\!\LC}
\newcommand{\sL}{\mathscr{L}}
\newcommand{\sM}{\mathscr{M}}
\newcommand{\sP}{\mathscr{P}}
\newcommand{\sQ}{\mathscr{Q}}
\newcommand{\sT}{\mathscr{T}}
\newcommand{\sY}{\mathscr{Y}}
\newcommand{\scx}{\cir\times \Cx}
\newcommand{\sqmo}{\sqrt{-1}}
\newcommand{\ssm}{\text{\scriptsize $\mathrsfs{S}$}}
\newcommand{\tP}{\widetilde{P}}
\newcommand{\tT}{\widetilde{\Theta }}
\newcommand{\tatp}{\tau \mstrut _{t'}}
\newcommand{\tat}{\tau \mstrut _t}
\newcommand{\thS}[2]{\rS(#1;#2)}
\newcommand{\tpii}{\frac{1}{2\pi \sqmo}}
\newcommand{\tpi}{2\pi \sqmo}
\newcommand{\tps}{\tilde{\psi}}
\newcommand{\tsm}{\text{\scriptsize $\mathrsfs{S}$}}
\newcommand{\wsp}[1]{\omega_{\textnormal{spin}}(#1)}

\begin{document}

\abovedisplayskip18pt plus4.5pt minus9pt
\belowdisplayskip \abovedisplayskip
\abovedisplayshortskip0pt plus4.5pt
\belowdisplayshortskip10.5pt plus4.5pt minus6pt
\baselineskip=15 truept
\marginparwidth=55pt

\makeatletter
\renewcommand{\tocsection}[3]{%
  \indentlabel{\@ifempty{#2}{\hskip1.5em}{\ignorespaces#1 #2.\;\;}}#3}
\renewcommand{\tocsubsection}[3]{%
  \indentlabel{\@ifempty{#2}{\hskip 2.5em}{\hskip 2.5em\ignorespaces#1%
    #2.\;\;}}#3} 
\def\@makefnmark{%
  \leavevmode
  \raise.9ex\hbox{\fontsize\sf@size\z@\normalfont\tiny\@thefnmark}} 
\def\multfoot{\textsuperscript{\tiny\color{red},}}
\def\footref#1{$\textsuperscript{\tiny\ref{#1}}$}
\makeatother

\setcounter{tocdepth}{2}




 \title[\today]{The dilogarithm and abelian Chern-Simons} 
 \author[D. S. Freed]{Daniel S.~Freed}
 \thanks{This material is based upon work supported by the National Science
Foundation under Grant Numbers DMS-1611957 and DMS-1711692.  We also thank
the Aspen Center for Physics, which is supported by National Science
Foundation grant PHY-1607611, as well as the Mathematical Sciences Research
Institute, which is supported by National Science Foundation Grant 1440140,
both of which provided support to the authors while this work was completed.}
 \address{Department of Mathematics \\ University of Texas \\ Austin, TX
78712} 
 \email{dafr@math.utexas.edu}
 \author[A. Neitzke]{Andrew Neitzke}
 \address{Department of Mathematics \\ Yale University \\ 10 Hillhouse Avenue\\
New Haven, CT 06511} 
 \email{neitzke@math.utexas.edu}
 \date{January 21, 2021}
 \begin{abstract} 
 We construct the (enhanced Rogers) dilogarithm function from the spin
Chern-Simons invariant of $\Cx$-connections.  This leads to geometric proofs
of basic dilogarithm identities and a geometric context for other properties,
such as the branching structure.
 \end{abstract}
\maketitle



{\small
\def\reftext{References}
\renewcommand{\tocsection}[3]{%
  \begingroup 
   \def\tmp{#3}%
   \ifx\tmp\reftext
  \indentlabel{\phantom{1}\;\;} #3%
  \else\indentlabel{\ignorespaces#1 #2.\;\;}#3%
  \fi\endgroup}
\tableofcontents
}

   \section{Introduction}\label{sec:1}
 
Investigations of the dilogarithm function, in its simplest form defined via
the power series 
  \begin{equation}\label{eq:76}
     \Li_2(z) = \sum\limits_{n=1}^{\infty} \frac{z^n}{n^2}, 
  \end{equation}
date back to Leibniz, Bernoulli, and Euler; one early reference is the 1809
essay of William Spence~\cite{S}.  In recent times the dilogarithm makes
appearances in hyperbolic geometry, algebraic $K$-theory, conformal field
theory, and beyond.  It and its relatives are the subject of survey articles,
such as~\cite{G,K,Z} which provide a wealth of references to the literature.
In their study of the scissors congruence problem, Dupont-Sah~\cite{DS} and
subsequently Dupont~\cite{D} relate a variant of the dilogarithm
function~\eqref{eq:76} to a Chern-Cheeger-Simons~\cite{CS,ChS} characteristic
class of flat principal $\SL_2\CC$-bundles.  As part of our study~\cite{FN}
of ``stratified abelianization'' of flat $\SL_2\CC$-connections on
3-manifolds, we discovered a geometric construction of this \emph{enhanced
Rogers dilogarithm} using \emph{abelian} Chern-Simons theory of flat
$\Cx$-connections on a 2-dimensional torus.  In this paper we present our
construction, and we use it to give geometric proofs of the basic dilogarithm
identities.
 
We begin in~\S\ref{sec:2} with an exposition of Chern-Simons invariants of
$\Cx$-connections.  Our work uses a square root for spin manifolds, which we
outline in~\S\ref{sec:3} and develop with proofs in Appendix~\ref{sec:6}.
The construction of the dilogarithm is carried out in~\S\ref{sec:4} and the
identities are proved in~\S\ref{sec:5}.  An inspiration for our construction,
independent of stratified abelianization, is a heuristic computation
motivated by topological string theory, as we explain in
Appendix~\ref{sec:7}.

We thank Ian Agol, Sasha Goncharov, and Charlie Reid for comments and
discussion.

   \section{Classical Chern-Simons theory}\label{sec:2}

The Chern-Simons invariant is defined for pairs~$(G,\lambda )$ consisting of
a real Lie group~$G$ with finitely many components and a class $\lambda \in
H^4(BG;\ZZ)$.  In this section we focus on the special case $(\Cx,{c_1}^2)$,
where $c_1\in H^2(B\Cx;\ZZ)$ is the universal first Chern class.  Here we
briefly indicate the constructions for trivializable $\Cx$-bundles.  In
Appendix~\ref{sec:6} we provide a general construction based on differential
cohomology which applies to all principal $\Cx$-bundles.  The treatment in
Appendix~\ref{sec:6} emphasizes the spin refinement of the Chern-Simons
invariant, though it can be adapted to the non-spin case discussed in this
section.  The references~\cite{F1,F2} contain more exposition and details.
The basic properties of classical Chern-Simons theory are compactly expressed
in the language of field theory (Theorem~\ref{thm:s3}).
 
Let $W$~be a closed oriented 4-manifold and $P\to W$ a principal
$\Cx$-bundle.  Then 
  \begin{equation}\label{eq:1}
     \left\langle {c_1}(P)^2,[W] \right\rangle\; \in \ZZ 
  \end{equation}
is a \emph{primary topological invariant} of $P\to W$.
Chern-Simons~\cite{CS} construct a \emph{secondary geometric invariant} of
$\Cx$-bundles with connection as follows.  Let $M$~be an arbitrary smooth
manifold and $\pi \:P\to M$ a principal $\Cx$-bundle.  Let $\Theta \in \OC
1P$ be a connection and $\Omega \in \OC 2M$ its curvature, i.e., $\pi ^*\Omega
=d\Theta $.  Define 
  \begin{align}
      \omega(\Theta ) &=-\frac{1}{4\pi ^2}\, \Omega \wedge \Omega  \; \in \OC4M \label{eq:2}\\ 
      \alpha(\Theta )  &=-\frac{1}{4\pi ^2}\, \Theta \wedge \Omega \; \in \OC3P. \label{eq:3}
  \end{align}
Then $d\omega =0$ and $d\alpha =\pi ^*\omega $.  These are the
\emph{Chern-Weil} and \emph{Chern-Simons} forms, respectively.  For $M=W$
a closed oriented 4-manifold, we have 
  \begin{equation}\label{eq:4}
     \int_{W}\omega(\Theta ) =\left\langle c_1(P)^2,[W] \right\rangle ; 
  \end{equation}
in particular, the left-hand side is independent of the connection~$\Theta
$.  
 
Suppose $M=X$ is a closed oriented 3-manifold and $\pi \:P\to X$ is
trivializable.  For each section~$s$ of~$\pi $ define 
  \begin{equation}\label{eq:5}
     \Gamma (\Theta ,s) = \int_{X} s^*\alpha (\Theta ). 
  \end{equation}
If $s'=s\cdot g$ for $g\:X\to\Cx$, then 
  \begin{equation}\label{eq:6}
     (s\cdot g)^*\alpha (\Theta ) = s^*\alpha (\Theta ) - \frac{1}{4\pi
     ^2}\,d\!\left( s^*\Theta \wedge \frac{dg}{g} \right).
  \end{equation}
Therefore, by Stokes' theorem $\Gamma (\Theta ,s)$~is independent of~$s$.
The Chern-Simons invariant is
  \begin{equation}\label{eq:8}
     \F X{\Theta } = \exp\left( \tpi\,\Gamma (\Theta ,s) \right) \; \in
     \Cx.
  \end{equation}

  \begin{remark}[]\label{thm:1}
 An alternative definition uses the fact that $P\to X$ can be written as the
boundary of a principal $\Cx$-bundle $\tP\to W$ over a compact oriented
4-manifold~$W$ with $\partial W=X$.  The connection~$\Theta $ extends to a
connection~$\tT$ on $\tP\to W$, and by Stokes' theorem the right hand side
of~\eqref{eq:5} equals
  \begin{equation}\label{eq:7}
     \int_{W}\omega (\tT). 
  \end{equation}
This definition works for any (possibly nontrivializable) $P\to X$ with
connection, but \eqref{eq:7}~is only independent of the choice of the
extension $\tP\to W$ modulo integers; see~\eqref{eq:4}.
  \end{remark}

If $X$~is a compact oriented 3-manifold with $\partial X=Y$, then $\Gamma
(\Theta ,s)$~is not independent of~$s$; the exact term in~\eqref{eq:6}
produces a boundary correction.  The dependence of $\exp\left( \tpi\,\Gamma
(\Theta ,s) \right)$ on~$s$ is encoded as follows.  Set $\rho \:Q=P\res
Y\to Y$ and $\eta =\Theta \res Q$.  Under our hypothesis the space
$\Sect(\rho )$ of sections of $\rho \:Q\to Y$ is nonempty; it is a torsor
over $\Map(Y,\Cx)$.  Define the complex line
  \begin{multline}\label{eq:9}
     \F Y{\eta } = \biggl\{ f\:\Sect(\rho )\to\CC : f(t\cdot h) = \exp\left(
     -\frac{\sqmo}{2\pi }\int_{Y}t^*\eta \wedge \frac{dh}{h}\,
     \right)f(t)
     \\[1pc]  \textnormal{for all }t\in \Sect(\rho ),\;  h\in\Map(Y,\Cx)\biggr\}.
  \end{multline}
Then 
  \begin{equation}\label{eq:10}
     \F X{\Theta } := \exp\left( \tpi\,\Gamma (\Theta ,-) \right)\; \in \F
     Y{\eta } 
  \end{equation}
is a well-defined nonzero element of the line~$\F Y{\eta }$.

  \begin{remark}[]\label{thm:2}
 By construction, a trivialization~$t$ of $\rho \:Q\to Y$ induces a
trivialization $\CC\xrightarrow{\;\cong \;}\F Y{\eta }$, i.e., a nonzero
element $\tat\in \F Y{\eta }$. 
  \end{remark}

The formal properties are best summarized in field theory language.  Let
$\brdn$ denote the category\footnote{For many reasons, among others to
construct smooth gluings, the objects are germs of 3-dimensional data over
the 2-dimensional data; see~\cite{Se}.  Our language in the text is a
simplification.  Also, it is crucial to extend to parametrized families of
$\Cx$-connections, for example to capture the variation formula for
Chern-Simons invariants, as we recount below.  See~\cite{ST} for field
theories sheafified over smooth manifolds.\label{foot:1}} whose objects are
closed oriented 2-manifolds~$Y$ equipped with a $\Cx$-connection~$\Theta _Y$;
a morphism $(Y_0,\Theta _0)\to (Y_1,\Theta _1)$ is a compact oriented
3-manifold~$X$ equipped with a $\Cx$-connection~$\Theta _X$, a diffeomorphism
$-Y_0\amalg Y_1\xrightarrow{\;\cong \;}\partial X$ (an oriented manifold~$M$
has a canonical reflection~$-M$ with the opposite orientation), and a
compatible isomorphism $\Theta _0\amalg \Theta _1\xrightarrow{\;\cong
\;}\partial \Theta _X$.  Composition of morphisms is gluing of bordisms, and
disjoint union provides a symmetric monoidal structure.  Let $\LC$ denote the
groupoid whose objects are 1-dimensional complex vector spaces and morphisms
are invertible linear maps; tensor product provides a symmetric monoidal
structure.

  \begin{theorem}[]\label{thm:s3}
 The exponentiated Chern-Simons invariant is a symmetric monoidal functor 
  \begin{qedequation}\label{eq:s16}
     \rF \:\brdn\longrightarrow \LC. \qedhere
  \end{qedequation}
  \end{theorem}

\noindent
 In other words, $\rF$~is an \emph{invertible field theory}, often called
\emph{classical Chern-Simons theory}. 

As mentioned in footnote~\footref{foot:1}, one can extend~$\rF$ to a theory
defined on smooth families.  Thus if $\sY\to S$ is a fiber bundle with fibers
closed oriented 2-manifolds, and $\sQ\to Y$ is a principal $\Cx$-bundle with
connection~$\eta $, then the Chern-Simons theory produces $\F{\sY/S}{\eta
}\to S$, a complex line bundle with covariant derivative.  Its curvature is
  \begin{equation}\label{eq:11}
     \curv\left( \F{\sY/S}{\eta }\longrightarrow S \right) =
     \frac{\sqmo}{2\pi}\int_{\,\sY/S}\Omega (\eta )\wedge \Omega (\eta ),
  \end{equation}
where $\Omega (\eta )\in \OC2{\sY}$ is the curvature of the
$\Cx$-connection~$\eta $.  Parallel transport along a path $\gamma \:[0,1]\to
S$ is the value of~$\F{\gamma ^*\sY}{\gamma ^*\eta }$ on the pullback
connection~$\gamma ^*\eta $ on $\gamma ^*\sQ\to\gamma ^*\sY$.  In particular,
holonomies of $\F{\sY/S}{\eta }\to S$ are computed as values of~$\rF$ on
mapping tori.

  \begin{remark}[]\label{thm:3}
 Cheeger-Simons~\cite{ChS} introduce differential characters and write the
Chern-Simons invariant of a closed oriented 3-manifold in those terms.  The
entire theory~$\rF$ fits into the theory of differential
cohomology~\cite{HS,BNV}, beginning with a differential lift of ${c_1}^2\in
H^4(B\Cx;\ZZ)$.  We develop this idea in Appendix~\ref{sec:6}.
  \end{remark}

Our application of~$\rF$ in~\S\ref{sec:4} is to families of \emph{flat}
$\Cx$-connections.  In the remainder of this section we observe some
properties which reflect the \emph{topological} nature of~$\rF$ on flat
connections.  
 
First, a principal $\Cx$-bundle $P\to M$ over any smooth manifold~$M$ admits
a flat connection if and only if $c_1(P)\in H^2(M;\ZZ)$ has finite order.  In
particular, if $H_1(M)$~is torsionfree, then only trivializable principal
$\Cx$-bundles on~$M$ admit flat connections.  The equivalence class of a flat
connection on $P\to M$ is a lift $[\Theta ]\in H^1(M;\CZ)$ of $c_1(P)\in
H^2(M;\ZZ)$ in the long exact sequence of cohomology groups induced from the
short exact sequence $\ZZ\to\CC\to\CZ$ of coefficients.

\newpage
  \begin{theorem}[]\label{thm:4}
 \

 \begin{enumerate}[label=\textnormal{(\roman*)}]

 \item Let $\Theta $~be a flat connection on a principal $\Cx$-bundle $\pi
\:P\to X$ over a closed oriented 3-manifold.  Then its Chern-Simons invariant
is
  \begin{equation}\label{eq:12}
     \F X{\Theta } = \exp\Bigl( \tpi \left\langle \,[\Theta ]\smile
     c_1(P)\,,\,[X] 
     \,\right\rangle \Bigr),\qquad [\Theta ]\in H^1(M;\CZ). 
  \end{equation}

 \item Let $\eta $~be a flat connection on a principal $\Cx$-bundle $\rho
\:Q\to Y$ over a closed oriented 2-manifold.  Then the trivialization
$\tat\in \F Y{\eta }$ in Remark~\ref{thm:2} depends only on the homotopy
class of the section~$t$ of~$\rho $.

 \item Let $\Theta $~be a flat connection on a principal $\Cx$-bundle $\pi
\:P\to X$ over a compact oriented 3-manifold with boundary.  Suppose $s$~is a
section of~$\pi $.  Then $\F X{\Theta }=\tau \mstrut _{\partial s}$
in~$\F{\partial X}{\partial \Theta }$.  
 \end{enumerate}
  \end{theorem}

  \begin{proof}
 Part~(i) is easy unless $\pi $~does not admit a section, in which case the
techniques here do not apply; we supply a proof at the end of
Appendix~\ref{sec:6}.  For~(ii) observe that the integrand in~\eqref{eq:9} is
the product of two closed 1-forms, so only depends on their de Rham
cohomology classes.  The generalization of~(iii) to arbitrary
connections~$\Theta $ is
  \begin{equation}\label{eq:13}
     \F X{\Theta } = \exp\left( \tpi\int_{X}s^*\alpha (\Theta ) \right) \tau
     \mstrut _{\partial s}, 
  \end{equation}
which is essentially the construction of~\eqref{eq:10}.  For flat~$\Theta $
the Chern-Simons form~$\alpha (\Theta )$ vanishes.
  \end{proof}

   \section{The spin refinement}\label{sec:3}

On spin manifolds Chern-Simons theory refines to a theory~$\rS$ with
$\rS^{\otimes 2}\cong \rF$.  The extra factor of~2 flows from that in the
primary invariant~\eqref{eq:1} on spin manifolds: if $P\to W$ is a principal
$\Cx$-bundle over a closed spin 4-manifold~$W$, then 
  \begin{equation}\label{eq:14}
     \frac 12 \left\langle {c_1}(P)^2,[W] \right\rangle\; \in \ZZ . 
  \end{equation}
In this section we state the properties of~$\rS$ we need; proofs are deferred
to Appendix~\ref{sec:6}. 

  \begin{remark}[]\label{thm:5}
 The function 
  \begin{equation}\label{eq:15}
     \begin{aligned} H^2(W;\ZZ)&\longrightarrow \qquad \ZZ \\ x&\longmapsto \frac
      12\left\langle x\smile x,[W] \right\rangle \end{aligned} 
  \end{equation}
is a quadratic refinement of the bihomomorphism 
  \begin{equation}\label{eq:16}
     \begin{aligned} H^2(W;\ZZ)\times H^2(W;\ZZ)&\longrightarrow\,\quad  \ZZ \\
      x\qquad ,\qquad y\qquad &\longmapsto \left\langle x\smile y,[W] \right\rangle
      \end{aligned} 
  \end{equation}
There is a compatible quadratic form on $H^1(X;\CZ)$ for $X$~a closed spin
3-manifold~\cite{MS,LL}; it computes the value of~$\thS X\Theta $ for flat
connections~$\Theta $.
  \end{remark}

As in Theorem~\ref{thm:s3} the formal properties are summarized in the
statement that $\rS$~is an invertible field theory (which can be evaluated on
fiber bundles).  Manifolds in the domain bordism category of~$\rS$ carry a
spin structure.  The codomain of~$\rS$ is the Picard groupoid~$\sLC$ whose
objects are $\zt$-graded (super) lines and whose morphisms are even
isomorphisms of super lines.  The monoidal structure is tensor product, and
the symmetry implements the Koszul sign rule.

  \begin{theorem}[]\label{thm:6}
 Spin Chern-Simons theory is a symmetric monoidal functor  
  \begin{equation}\label{eq:17}
     \rS \:\brds\longrightarrow \sLC. 
  \end{equation}
There is an isomorphism $\rS^{\otimes 2}\cong \rF$. 
  \end{theorem}

  \begin{remark}[]\label{thm:10}
 The invariant of a $\Cx$-connection over a closed spin 3-manifold has a
description analogous to that in Remark~\ref{thm:1}.  In that case we must
bound~$X$ by a compact \emph{spin} 4-manifold and put the factor~$1/2$ in the
integral~\eqref{eq:7}.
  \end{remark}

For convenience we state further properties of~$\rS$ simultaneously for
families of 2-~and 3-manifolds.  Consider the sequence
  \begin{equation}\label{eq:s59}
     P\xrightarrow{\;\;\pi \;\;} M_\sigma \xrightarrow{\;\;p\;\;}
     S 
  \end{equation}
of smooth maps in which $p$~is a fiber bundle of smooth manifolds, or of
bordisms; $\pi $~is a principal $\Cx$-bundle with connection $\Theta \in
\OC1P$; and $\sigma $~is a spin structure on the relative tangent bundle
$T(M/S)\to S$.  Let $\dim(p)\:M\to \ZZ^{\ge0}$ be the locally constant
function whose value at~$m\in M$ is the dimension of the relative tangent
space~$T_m(M/S)$.  Let $\Omega(\Theta ) \in \Omega ^2_M(\CC)$ be the
curvature of~$\Theta $.  The following is proved in Appendix~\ref{sec:6}.

  \begin{theorem}[]\label{thm:s20}
 \

 \begin{enumerate}[label=\textnormal{(\roman*)}]

 \item If $\dim(p)=3$ and the fibers of~$p$ are closed, then $\thS{M_\sigma/S
}{\Theta}\:S\to\Cx$  satisfies
  \begin{equation}\label{eq:s60}
     \frac{d\thS{M_\sigma/S}{\Theta}}{\thS{M_\sigma/S }{\Theta}} =
     -\frac{\sqmo}{4\pi 
     }\int_{M/S} \Omega(\Theta ) \wedge \Omega(\Theta ) .
  \end{equation}

 \item Continuing, if $\sigma '$~is a spin structure whose difference
with~$\sigma $ represents a class $\delta \in H^1(M;\zt)$, then the ratio of
spin Chern-Simons invariants is 
  \begin{equation}\label{eq:s61}
     \frac{\thS{M_{\sigma '}/S}{\Theta}}{\thS{M_{\sigma }/S}{\Theta}} =
     (-1)^{p_*(\delta \smile \overline{c_1}(P) )}, 
  \end{equation}
where $\overline {c_1}$~is the mod~2 reduction of the first Chern class.

 \item If $\dim(p)=3$ and the fibers of~$p$ are compact with boundary, then
$\thS{M_{\sigma }/S}{\Theta}$ is a section of the even complex line bundle
$\thS{\partial M_{\sigma }/S}{\Theta}\to S$; its covariant derivative is
  \begin{equation}\label{eq:40}
     \nabla \thS{M_{\sigma}/S}{\Theta} = \left[ -\frac{\sqmo}{4\pi 
     }\int_{M/S} \Omega(\Theta ) \wedge \Omega(\Theta ) \right]
     \thS{M_{\sigma }/S}{\Theta}. 
  \end{equation}

 \item Continuing, suppose $s\:M\to P$ is a section of~$\pi $.  Its
restriction~$\partial s$ to~$\partial M$ induces a trivialization $\tau
\mstrut _{\partial s}$ of $\thS{\partial M_{\sigma}/S}{\Theta}\to S$.  Then 
  \begin{equation}\label{eq:41}
     \thS{M_{\sigma}/S}{\Theta} = \exp\left( -\frac{\sqmo}{4\pi 
     }\int_{M/S}s^*\Theta\wedge \Omega (\Theta )
      \right) \tau \mstrut _{\partial s} . 
  \end{equation}

 \item If $\dim(p)=2$ and the fibers of~$p$ are closed, then $\thS{M_{\sigma
}/S}{\Theta}\to S$ is a complex super line bundle with covariant
derivative; its curvature is
  \begin{equation}\label{eq:s62}
     \curv\left(\thS{M_{\sigma}/S}{\Theta}\longrightarrow  S\right)=
     \frac{\sqmo}{4\pi 
     }\int_{M/S} \Omega(\Theta ) \wedge \Omega(\Theta ) .  
  \end{equation}
Its $\zt$-grading is $p_*\bigl[\overline{c_1}(P) \bigr]\:S\to\zt$, where
$\overline{c_1}$~is the mod~2 reduced first Chern class.

 \item Continuing, a section $t\:M\to P$ of~$\pi $ induces a trivialization
$\tat\:S\to L$ of the Chern-Simons line bundle, relative to which the
connection form is
  \begin{equation}\label{eq:s63}
     \frac{\nabla \tat}{\tat} = \frac{\sqmo}{4\pi }\int_{M/S} t^*\Theta \wedge
     \Omega (\Theta ). 
  \end{equation}

 \item Continuing, given $h\:M\to\Cx$ set $t'=t\cdot h\:M\to P$.  Then
  \begin{equation}\label{eq:s64}
     {\tatp} = \epsilon \mstrut _{t,h}\exp\left( -\frac{\sqmo}{4\pi }
     \int_{M/S} t^*\Theta\wedge \frac{dh}{h} \right) {\tat} , 
  \end{equation}
where 
  \begin{equation}\label{eq:s65}
     \epsilon \mstrut _{t,h}(s)=(-1)^{\sigma _s([h_s])},\qquad s\in S.
  \end{equation}
Here $\sigma _s\:H^1\bigl(p\inv (s);\zt\bigr)\to\zt$ is the quadratic
refinement of the intersection pairing given by the spin structure on the
fiber~$p\inv (s)$ , and $[h_s]\in H^1\bigl(p\inv (s);\zt\bigr)$ is the
reduction modulo two of the homotopy class of~$h\res{p\inv (s)}$.

 \end{enumerate} 
  \end{theorem}

  \begin{corollary}[]\label{thm:s27}
 In the situation of Theorem~\ref{thm:s20}\,\textnormal{(vii)}, if $\Theta $~is
flat then \eqref{eq:s64}~only depends on the homotopy class of~$h$.
  \end{corollary}

  \begin{proof}
  The factor~$\epsilon\mstrut _{t,h}$ in~\eqref{eq:s65} depends only on the
homotopy class of~$h$ modulo two.  If $\Theta $~is flat, then $t^*\Theta $~is
a closed 1-form, and the integral in~\eqref{eq:s64} reduces to the cohomology
pushforward $p_*\:H^2(M;\CC)\to H^0(S;\CC)$ applied to $ [t^*\Theta
]\smile[h]$, as in Theorem~\ref{thm:4}(ii).
  \end{proof}

In the setup of Theorem~\ref{thm:s20}, suppose given two families over the
same base~$S$ and maps 
  \begin{equation}\label{eq:52}
     \begin{gathered} \xymatrix@C-3ex@R+1ex{P\ar[rr]^{\tps} \ar[d]_{\pi } &&
     P'\ar[d]^{\pi '} \\ M_\sigma \ar[rr]^{\psi }\ar[dr] && M'_{\sigma
     '}\ar[dl]\\ &S} 
     \end{gathered} 
  \end{equation}
such that $\psi $~is a diffeomorphism of spin manifolds, $\tps$~is an
isomorphism of principal $\Cx$-bundles, and $\tps$~preserves the connections: 
$\tps^*\Theta '=\Theta $.

  \begin{theorem}[]\label{thm:17}
 \

 \begin{enumerate}[label=\textnormal{(\roman*)}]

 \item If $\dim(p)=3$ and the fibers of~$p$ are closed, then $\thS{M'_{\sigma
'}/S}{\Theta'} = \thS{M_\sigma /S}{\Theta}$.

 \item If $\dim(p)=2$ and the fibers of~$p$ are closed, then there is a flat
isomorphism of spin Chern-Simons line bundles 
  \begin{equation}\label{eq:53}
     \begin{gathered} \xymatrix@C-3ex@R+1ex{\thS{M_\sigma
     /S}{\Theta}\ar[dr]\ar[rr]^\Psi 
     &&\thS{M'_{\sigma '}/S}{\Theta '}\ar[dl] \\&S} \end{gathered} 
  \end{equation}

 \item Continuing, if $t\:M\to P$ and $t'\:M'\to P'$ are sections of~$\pi
,\pi '$ such that $\tps\circ t=t'\circ \psi $, then the induced
trivializations~$\tat,\tatp$ of the line bundles in~\eqref{eq:53}
satisfy $\tatp=\Psi \circ \tat$.

 \end{enumerate} 
  \end{theorem}

\noindent
 We leave the reader to formulate and prove (1)~functoriality for the case
$\dim(p)=3$ and the fibers of~$p$ are compact manifolds with boundary, and
(2)~the behavior of~$\rS$ under reversal of orientation.  Theorem~\ref{thm:17}
follows from the constructions in Appendix~\ref{sec:6}.

   \section{The dilogarithm from abelian Chern-Simons}\label{sec:4}

In~\S\ref{subsec:4.1} we use the spin Chern-Simons theory of~\S\ref{sec:3} to
construct a holomorphic function~$L$ on an abelian cover of the thrice
punctured complex projective line.  We identify it with the enhanced Rogers
dilogarithm in~\S\ref{subsec:4.2}.

  \subsection{Construction of~$L$}\label{subsec:4.1}

Fix the standard torus 
  \begin{equation}\label{eq:n14}
     T = \RR^2\!\!\bigm/\!\ZZ^2 
  \end{equation}
and standard coordinates~$\theta ^1,\theta ^2$ on~$\RR^2$.  The first
homology group~$H_1(T)$ has generators the coordinate loops $t\mapsto(t,0)$ and
$t\mapsto(0,t)$, $0\le t\le 1$, in $(\theta ^1,\theta ^2)$ coordinates.  Let
$\sigma $~be the spin structure on~$T$ characterized by the property that the
coordinate loops inherit the bounding spin structure.  Introduce
  \begin{equation}\label{eq:21}
     \begin{aligned} \hMT&=\CC^2 \\ \MT&=(\Cx)^2\end{aligned} 
  \end{equation}
with standard coordinates $u_1,u_2$ and $\mu _1,\mu _2$, respectively.
Define the principal $\ZZ^2$-bundle
  \begin{equation}\label{eq:19}
  \begin{aligned}
     e\:\;\hMT\quad &\longrightarrow \quad\MT \\ 
     (u_1,u_2)&\longmapsto (e^{u_1},e^{u_2})
  \end{aligned}
  \end{equation}
The notation~\eqref{eq:21} is deliberately evocative of moduli spaces; see
Remark~\ref{thm:7} below.

The trivial $\Cx$-bundle 
  \begin{equation}\label{eq:20}
     \hr \:\hMT\times T\times \Cx\longrightarrow \hMT\times T 
  \end{equation}
with section~$\hto$ carries a complex connection form~$\he $ characterized
by
  \begin{equation}\label{eq:n17}
     \hto^*\he  = -u_1\,d\theta ^1 - u_2\,d\theta ^2\;\in \Omega
     ^1_{\hMT\times T}(\CC) .
  \end{equation}
The action of~$\ZZ^2$ on the base of~\eqref{eq:20} lifts to the total
space:
  \begin{multline}\label{eq:n20}
     (n_1,n_2)\cdot (u_1,u_2,\theta ^1,\theta ^2,\lambda )\\=
     \Bigl(u_1 + \tpi n_1,u_2+\tpi n_2,\theta ^1,\theta ^2,\exp
     \bigl[\tpi(n_1\theta ^1+n_2\theta ^2)\bigr]\lambda \Bigr), 
  \end{multline}
where $(n_1,n_2)\in \ZZ^2$ and $\lambda \in \Cx$; the connection form~$\he $
is preserved.  Hence the $\Cx$-bundle~\eqref{eq:20} with connection descends
to a $\Cx$-bundle
  \begin{equation}\label{eq:n21}
     \rho \:\sQ\longrightarrow \MT\times T
  \end{equation}
with connection~$\eta $.   Its curvature is the differential
of~\eqref{eq:n17}: 
  \begin{equation}\label{eq:22}
     \Omega (\eta )=-\frac{d\mu _1}{\mu _1}\wedge d\theta ^1 - \frac{d\mu
     _2}{\mu _2}\wedge d\theta ^2. 
  \end{equation}
This formula shows that the bundle~$\rho $ in~\eqref{eq:n21} is topologically
nontrivial.  (The base $\MT\times T$ deformation retracts to a 4-torus; the
restrictions to two sub 2-tori have nonzero first Chern class.)

  \begin{remark}[]\label{thm:7}
 The connection~ $\eta $ on~$\rho $ in~\eqref{eq:n21} defines a
\emph{universal} family of \emph{flat} $\Cx$-connections over~$T$.  Namely,
$\eta $~is flat on the fibers of the projection to~$\MT$, and its holonomies
at~$(\mu _1,\mu _2)\in \MT$ about the standard cycles on~$T$ are~$\mu _1 ,\mu
_2 $, as follows from~\eqref{eq:n17}.\footnote{The holonomy is the
exponential of \emph{minus} the integral of the connection form.  The signs
are chosen so that \eqref{eq:27}~ matches the differential of the
dilogarithm.}  Furthermore, the pullback under~\eqref{eq:19} is a universal
family of flat $\Cx$-connections over~$T$ equipped with a homotopy class of
trivializations of the underlying $\Cx$-bundle.  This explains the moduli
space notation~\eqref{eq:21}.  Also, $\MT$~is a holomorphic symplectic
manifold; see~\eqref{eq:sd24} below.
  \end{remark}

Define the submanifolds
  \begin{equation}\label{eq:n15}
  \begin{aligned} 
     \hMTp &= \bigl\{(u_1,u_2)\in \hMT : e^{u_1} + e^{u_2}=1\bigr\}\\
     \MTp &= \bigl\{(\mu _1,\mu _2)\in \MT: \mu _1+\mu _2=1\bigr\} .
  \end{aligned}
  \end{equation}
Then  
  \begin{equation}\label{eq:28}
     \begin{aligned} \MTp\;\;&\longrightarrow \CP^1\setminus \{0,1,\infty \} \\
      (\mu _1,\mu _2)&\longmapsto \mu _1\end{aligned} 
  \end{equation}
is a diffeomorphism.  Also, \eqref{eq:19}~restricts to a principal
$\ZZ^2$-bundle $e'\:\hMTp\to\MTp$.  Then~$\rho $ in~\eqref{eq:n21} and its 
pullbacks and restrictions fit into a diagram of principal $\Cx$-bundles with
connection: 
  \begin{equation}\label{eq:23}
     \begin{gathered} \xymatrix@C-2ex{&\;(\epsilon')
     ^*\sQ'\ar@<-.5ex>[dl]_<<<<<<<{\hr'} && \;\epsilon
     ^*\sQ\ar@<-.5ex>[dl]_<<<<<<<{\hr}\\ \hMTp\times
     T{\ar@[red]@<-.5ex>[ur]_<<<<<<<<<{\red\hto'}}\ar[dd]^{\epsilon'}
     \;\ar@{^{(}->}[rr]  
     && \hMT\times T\ar@[red]@<-.5ex>[ur]_<<<<<<<<<{\red\hto}\ar[dd]^\epsilon \\
     &\sQ'\ar[dl]_<<<<<<<{\rho'} && 
     \sQ\ar[dl]_<<<<<<<{\rho}\\ \MTp\times T\;\ar@{^{(}->}[rr] && \MT\times
     T} \end{gathered}  
  \end{equation}
Here $\epsilon =e\times \id\mstrut _T$ and $\epsilon '=e'\times \id\mstrut
_T$\,. 

Apply the spin Chern-Simons theory~$\rS$ of~\S\ref{sec:3} to the four
$\Cx$-bundles with connection (two of them with trivialization)
in~\eqref{eq:23}.  Let
  \begin{equation}\label{eq:sd25}
     \sL\longrightarrow \MT
  \end{equation}
be the spin Chern-Simons line bundle $\thS{(\MT\times T)/\MT\,}{\,\eta
}\to\MT$ with its covariant derivative.  Observe that the $\Cx$-bundle~$\rho
$ restricted to $\mu _1\times \mu _2\times T$ is topologically trivial for
all $(\mu _1,\mu _2)\in \MT$, since the restriction carries a flat
connection, from which it follows that \eqref{eq:sd25}~is an \emph{even} line
bundle (see Theorem~\ref{thm:s20}(v)).  The Chern-Simons line
bundle~\eqref{eq:sd25} and its various pullbacks fit into the diagram
  \begin{equation}\label{eq:24}
     \begin{gathered} \xymatrix@C-2ex{&\;\;(e') ^*\sL'\ar@<-.5ex>[dl]
     && \;\;e ^*\sL\ar@<-.5ex>[dl]\\ \hMTp
     \ar@[red]@<-.5ex>[ur]_<<<<<<<<<{\red\htao'}\ar[dd]^{e'}
     \;\ar@{^{(}->}[rr] &&  
     \hMT\ar@[red]@<-.5ex>[ur]_<<<<<<<<<{\red\htao}\ar[dd]^e \\
     &\sL'\ar@<-.5ex>[dl]&& \sL\ar[dl]\\ \MTp
     \;\ar@[blue]@<-.5ex>[ur]_<<<<<<<<<{\blue\tau '}
     \ar@{^{(}->}[rr] && \MT} \end{gathered} 
  \end{equation}
We explain the section~$\tau '$ of $\sL'\to \MTp$ after the proof of the
following.

  \begin{proposition}[]\label{thm:s28}
 \

 \begin{enumerate}[label=\textnormal{(\roman*)}]

 \item The curvature of the covariant derivative on~\eqref{eq:sd25} is 
  \begin{equation}\label{eq:sd24}
     \tpii\;\frac{d\mu _1}{\mu _1}\wedge \frac{d\mu _2}{\mu _2}.
  \end{equation}
The line bundle  
  \begin{equation}\label{eq:s75}
     \sL'\longrightarrow \MTp 
  \end{equation}
is flat.

 \item The line bundle~\eqref{eq:s75} has trivial holonomy.

 \end{enumerate} 
  \end{proposition}

  \begin{proof}
 The curvature statement~(i) follows from~\eqref{eq:s62} and~\eqref{eq:22}.
We compute the holonomy about $(0,1)\in \MTp$ using the family of loops
(set $i=\sqmo$)
  \begin{equation}\label{eq:36}
  \begin{aligned}
     \Gamma _\epsilon \:[0,2\pi ]&\longrightarrow \quad \MTp\\
       t\quad &\longmapsto  (\epsilon e^{it},1-\epsilon e^{it})
  \end{aligned}
  \end{equation}
Since the curvature vanishes, the result is independent of~$\epsilon \in
(0,1/2)$.  The computation for the holonomy about $(1,0)\in \MTp$ follows by
symmetry $\mu _1\leftrightarrow\mu _2$.

The loop~\eqref{eq:36} lifts to the path $\hGe\:[0,2\pi ]\to\hMTp$ defined by 
  \begin{equation}\label{eq:37}
     \begin{aligned} u_1&=\log\epsilon +it \\ u_2&=\log(1-\epsilon
      e^{it}),\end{aligned} 
  \end{equation}
where for~$u_2$ choose the branch of the logarithm with $\log 1=0$.
Use~\eqref{eq:s63} to compute the connection form of $e^*\sL\to\hMT$ relative
to the trivialization~$\htao$ as
  \begin{equation}\label{eq:sd45}
     \frac{1}{4\pi \sqmo}(u_1du_2 - u_2du_1). 
  \end{equation}

To compute the holonomy integrate~\eqref{eq:sd45} along~$\hGe$, and then
use~\eqref{eq:s64} to correct for the change of trivialization between the
two endpoints $\bigl(\log\epsilon ,\log(1-\epsilon ) \bigr)$ and
$\bigl(\log\epsilon +2\pi i,\log(1-\epsilon ) \bigr)$.  Set $z=\mu
_1=e^{u_1}$ and compute
  \begin{equation}\label{eq:38}
     \int_{\hGe}u_1\,du_2 - u_2\,du_1 = -\int_{\Gamma _\epsilon
     }\frac{ \log\epsilon }{1- z}\,dz \;+\; \int_{0}^{2\pi }
     \frac{\epsilon te^{it}}{1-\epsilon e^{it}}\,dt \;-\; \int_{\Gamma
     _\epsilon }\frac{\log(1-\epsilon z)}{z}\,dz. 
  \end{equation}
The first and last terms vanish by Cauchy's theorem.  The integrand in the
second term has norm bounded above by~$4\pi \epsilon $, hence the integral
converges to zero as~$\epsilon \to0$.  

Turning to the change of trivialization, the gauge transformation is
multiplication by $h=e^{2\pi i\theta ^1}$ and so by~\eqref{eq:n17} the
integrand in~\eqref{eq:s64} is the 2-form
  \begin{equation}\label{eq:39}
     \bigl(\log\epsilon \,d\theta ^1 + \log(1-\epsilon )\,d\theta ^2
     \bigr)\wedge \bigl(\tpi\,d\theta ^1 \bigr) = -\tpi\log(1-\epsilon
     )\,d\theta ^1\wedge d\theta ^2; 
  \end{equation}
both it and its integral over~$T$ vanish in the limit~$\epsilon \to0$.
Finally, the quadratic function~\eqref{eq:s65} is~1 on the coordinate loop in
the $\theta ^1$-direction, by our choice\footnote{The value of the quadratic
function on the $\theta ^2$-coordinate loop enters the computation of
holonomy around $(0,1)\in \MTp$.  The holonomy of~\eqref{eq:s75} is not
trivial for any other spin structure on~$T$, but see Remark~\ref{thm:9}.} of
spin structure on~$T$.  Since the holonomy is independent of~$\epsilon $, and
the computation converges as~$\epsilon \to0$, that limit suffices to prove
that the holonomy is trivial.
  \end{proof}

It follows that \eqref{eq:s75}~admits a $\Cx$-torsor~$\sT'$ of flat nonzero
sections~$\tau '$.  The pullback $(e')^*\sL\to \hMTp$ has a canonical nonzero
section~$\htao'$ which is not flat; see~\eqref{eq:sd45}.  For each~$\tau '\in
\sT'$ define
  \begin{equation}\label{eq:25}
     \varphi =\frac{\htao'}{\tau '}\:\hMTp\longrightarrow \Cx. 
  \end{equation}
Varying $\tau '\in \sT'$ changes~$\varphi $ by a multiplicative constant.
From~\eqref{eq:sd45} compute
  \begin{equation}\label{eq:sd52}
     \frac{d\varphi }{\varphi } = \frac{1}{4\pi \sqmo}(u_1du_2 - u_2du_1)
  \end{equation}
for all~$\tau '\in \sT'$, where recall $e^{u_1} + e^{u_2}=1$.  Write
  \begin{equation}\label{eq:s76}
     \varphi = \exp\left(\frac{L}{\tpi}\right) 
  \end{equation}
to define a function 
  \begin{equation}\label{eq:26}
     L\:\hMTp\to \CZt.
  \end{equation}
Here we use the Tate twists
  \begin{equation}\label{eq:29}
  \begin{aligned}
     \Zo&=\tpi\ZZ\\
     \Zt&=\Zo^{\otimes 2}=4\pi ^2\ZZ .
  \end{aligned}
  \end{equation}
The function~$L$ is determined up to an additive constant: the choice
of~$\tau '\in \sT'$.  Any choice satisfies
  \begin{equation}\label{eq:27}
     dL = \frac{u_1du_2 - u_2du_1}{2}. 
  \end{equation}

  \begin{remark}[]\label{thm:9}
 Our construction uses a particular spin structure on the torus~$T$ and the
particular choice of lagrangian submanifold~$\MTp\subset \MT$.  There are
three other spin structures and correlated choices of lagrangians which lead
to variations of the enhanced Rogers dilogarithm.  We use all four functions
in~\cite{FN}. 
  \end{remark}

  \subsection{Dilogarithms}\label{subsec:4.2}

We refer to~\cite{G,K,Z} and the references therein for details about the
various dilogarithm functions.
 
The theory begins with the power series 
  \begin{align}
      -\log(1-z) &= \sum\limits_{n=1}^{\infty} \frac{z^n}{n} \label{eq:30}\\ 
      \Li_2(z) &= \sum\limits_{n=1}^{\infty} \frac{z^n}{n^2}, \label{eq:31}
  \end{align}
convergent for~$z\in \CC$ satisfying~$|z|<1$, and analytically continued to
$z\in \CC\setminus [1,\infty )$.  Equation~\eqref{eq:30} is an identity;
equation~\eqref{eq:31} defines the \emph{Spence dilogarithm}~$\Li_2$.
Differentiate the power series:
  \begin{equation}\label{eq:32}
     d\Li_2(z) \;=\; -\frac{\log(1-z)}{z}\,dz \;=\; -u_2\,du_1 \;=\;
     -\frac{u_2\,du_2}{1-e^{-u_2}}, 
  \end{equation}
using notation from~\eqref{eq:19}, \eqref{eq:28} and setting~$z=\mu _1$.  The
last expression is a meromorphic 1-form on the $u_2$-line with simple poles
at $\Zo\subset \CC$ and residues in~$\Zo$.  It follows that 
  \begin{equation}\label{eq:33}
     \begin{aligned} F\:\CC\setminus \Zo&\longrightarrow \CZt \\
      u_2\quad &\longmapsto \Li_2(1-e^{u_2})\end{aligned} 
  \end{equation}
is a well-defined function.  We lift it under the $\ZZ$-covering map 
  \begin{equation}\label{eq:34}
     \begin{aligned} \hMTp\;\;&\longrightarrow \CC\setminus \Zo \\
      (u_1,u_2)&\longmapsto \quad u_2\end{aligned} 
  \end{equation}
Then 
  \begin{equation}\label{eq:35}
     F(u_2) + \frac 12u_1u_2 \;=\; \Li_2(z) + \frac 12\log(z)\log(1-z)\quad
     \mod\Zt 
  \end{equation}
is a well defined function $\hMTp\to\CZt$, and from~\eqref{eq:32} its
differential equals~$dL$ in~\eqref{eq:27}.  Therefore,
\eqref{eq:35}~equals~$L$ up to an additive constant.  (Recall that in any
event $L$~is only defined up to an additive constant.)  The
function~\eqref{eq:35} is called the \emph{enhanced Rogers dilogarithm}.   It
is denoted~`$\widehat{D}$' in~\cite{Z}.
 
This discussion proves the following. 

  \begin{theorem}[]\label{thm:8}
 The function~$L$ in~\eqref{eq:26}, defined up to a constant using spin
Chern-Simons theory with gauge group~$\Cx$, equals the enhanced Rogers
dilogarithm. 
  \end{theorem}

   \section{Properties of the dilogarithm}\label{sec:5}
 
In this section we use the geometric construction of the enhanced Rogers
dilogarithm (\S\ref{subsec:4.1}) to derive a few of its standard properties
directly from spin Chern-Simons theory. 
 
Consider first the transformation law under a deck transformation of the
$\ZZ^2$-covering map $e'\:\hMTp\to\MTp$.  Fix $\tau '\in \sT'$ and so a
choice of~$\varphi ,L$; see~\eqref{eq:25} and~\eqref{eq:s76}. 

  \begin{theorem}[]\label{thm:11}
 For all $(u_1,u_2)\in \hMTp$ and $(n_1,n_2)\in \ZZ^2$,
  \begin{equation}\label{eq:42}
     L(u_1+2\pi in_1,u_2+2\pi in_2) = L(u_1,u_2) + \pi i(n_1u_2-n_2u_1) +
     2\pi ^2n_1n_2\qquad \mod\Zt. 
  \end{equation}
  \end{theorem}
 
\noindent 
 Equation~\eqref{eq:42} appears in~\cite[p.~25]{Z} and as~\cite[(A.18)]{APP}.

  \begin{proof}
 Since the section~$\tau '$ of $(e')^*\sL'\to\hMTp$ is pulled back via
$e'\:\hMTp\to\MTp$, the change in $\varphi =\htao'/\tau '$ under a deck
transformation of~$e'$ is due to the change in the trivialization~$\htao'$.
Use the general formula~\eqref{eq:s64} to compute.  From~\eqref{eq:n17} the
connection form relative to the trivialization at $(u_1,u_2)\in \hMTp$ is 
  \begin{equation}\label{eq:43}
     u_1\,d\theta ^1 + u_2\,d\theta ^2, 
  \end{equation}
and the relevant gauge transformation is multiplication by the function 
  \begin{equation}\label{eq:44}
     h=\exp\bigl(2\pi i(n_1\theta ^1 + n_2\theta ^2) \bigr). 
  \end{equation}
The integral in~\eqref{eq:s64} is then $\frac 12(n_1u_2-n_2u_1)$; with the
correct prefactor it contributes the second term in~\eqref{eq:42}.  The third
term derives from the factor~\eqref{eq:s65} in~\eqref{eq:s64}.  The quadratic
form~$\sigma $ for our choice of spin structure on~$T$ is
  \begin{equation}\label{eq:45}
     \sigma (n_1,n_2) = n_1n_2\pmod 2, 
  \end{equation}
where $(n_1,n_2)\in \ZZ^2$~is the homotopy class of~$h$ in~$H^1(T;\ZZ)$,
relative to the standard basis.
  \end{proof}
 
Next, we prove a reflection identity, which appears as~\cite[(A.21)]{APP}
and, for a restriction of~$L$, in~\cite[p.~23]{Z}.

  \begin{theorem}[]\label{thm:12}
 For the diffeomorphism 
  \begin{equation}\label{eq:46}
     \begin{aligned} \Psi \:\hMTp\;\;&\longrightarrow \;\;\hMTp \\
      (u_1,u_2)&\longmapsto (u_2,u_1)\end{aligned} 
  \end{equation}
the sum $L+\Psi ^*L$ is a constant function. 
  \end{theorem}

Of course, the result holds for all~$\tau '\in \sT'$. 

  \begin{proof}
 Consider the reflection diffeomorphism 
  \begin{equation}\label{eq:47}
     \begin{aligned} \psi \:T\;\;\;&\longrightarrow \;\;\;T \\ (\theta
     ^1,\theta ^2)&\longmapsto (\theta ^2,\theta ^1)\end{aligned} 
  \end{equation}
on the torus~$T$.  It reverses orientation, maps our chosen spin
structure~$\sigma $ to its opposite,\footnote{A spin structure on an
$n$-dimensional oriented Riemannian manifold~$M$ is encoded in a sequence 
  \begin{equation}\label{eq:82}
     P\xrightarrow{\;\;\rho \;\;}\SO(M)\xrightarrow{\;\;\pi \;\;}M 
  \end{equation}
in which $\pi $~is the principal $\SO_n$-bundle of oriented orthonormal
frames and $\pi \circ \rho $~is a principal $\Spin_n$-bundle .  The
\emph{opposite spin structure} $P'\xrightarrow{\;\rho
'\;}\SO'(M)\xrightarrow{\;\pi '\;}M$ is constructed by taking the complement
of~\eqref{eq:82} in 
  \begin{equation}\label{eq:83}
     P\times \mstrut _{\Spin_n}\Pin_n\longrightarrow \O(M)\longrightarrow
     M, 
  \end{equation}
where either pin group $\Pin^{\pm}_n$ can be used.  Here $\O(M)\to M$ is the
principal $\O_n$-bundle of orthonormal frames.} induces the reflection $(\mu
_1,\mu _2)\mapsto(\mu _2,\mu _1)$ on the moduli space~$\MT$ of flat
$\Cx$-connections, which lifts to the reflection $(u_1,u_2)\mapsto(u_2,u_1)$
on~$\hMT$, which in turn restricts to~\eqref{eq:46} on~$\hMTp$.  Use
formulas~\eqref{eq:n17} and~\eqref{eq:n20} to compatibly lift the involution
to the bundles in~\eqref{eq:23}; the lift preserves the universal
connections.  The functoriality of Chern-Simons (Theorem~\ref{thm:17})
implies that the involution lifts to the Chern-Simons line bundles with
covariant derivative in~\eqref{eq:24}, except because orientation on~$T$ is
reversed the bundle~$\sL$ is mapped to~$\sL\inv $.  It follows that $\Psi
^*\varphi =\varphi \inv $, up to a multiplicative constant.
  \end{proof}

  \begin{remark}[]\label{thm:13}
 A related identity states that $L(z)+L(1/z)$ is constant;
see~\cite[(A.24)]{APP} for the precise form on the cover~$\hMTp$.  It can be
proved by a similar method, but the relevant diffeomorphism of~$T$ does not
preserve the spin structure, so one needs to expand the theory as indicated
in Remark~\ref{thm:9}.
  \end{remark}

Finally, we prove the 5-term relation satisfied by the dilogarithm~\cite{Z},
\cite[(A.8)]{APP}.  For that we change notation and use~$u,v$ in place of
$u_1,u_2$ as the standard coordinates on $\hMT=\CC^2$.

  \begin{theorem}[]\label{thm:14}
 The sum $L(u_1,v_1) + \cdots + L(u_5,v_5)$ is independent of~$(u_i,v_i)\in
\hMTp$, $i\in \zmod5$, which satisfy 
  \begin{equation}\label{eq:48}
     v_i = u_{i-1} + u_{i+1}\qquad \textnormal{for all $i$}. 
  \end{equation}
  \end{theorem}

  \begin{remark}[]\label{thm:15}
 Write $z_i=e^{u_i}$ and assume $e^{u_i}+e^{v_i}=1$.  Then equation~ \eqref{eq:48}
implies
  \begin{equation}\label{eq:49}
     1-z_i = z_{i-1}z_{i+1}. 
  \end{equation}
There is a connected complex 2-manifold~$\MXp$ of solutions 
  \begin{equation}\label{eq:50}
     x\;,\;\frac{1-x}{1-xy}\;,\; \frac{1-y}{1-xy}\;,\; y\;,\; 1-xy 
  \end{equation}
to~\eqref{eq:49}, parametrized by $x,y\in \CC$ satisfying $xy\neq 0$, $x\neq
1$, $y\neq 1$, and $xy\neq 1$.
  \end{remark}

  \begin{figure}[ht]
  \centering
  \includegraphics[scale=.7]{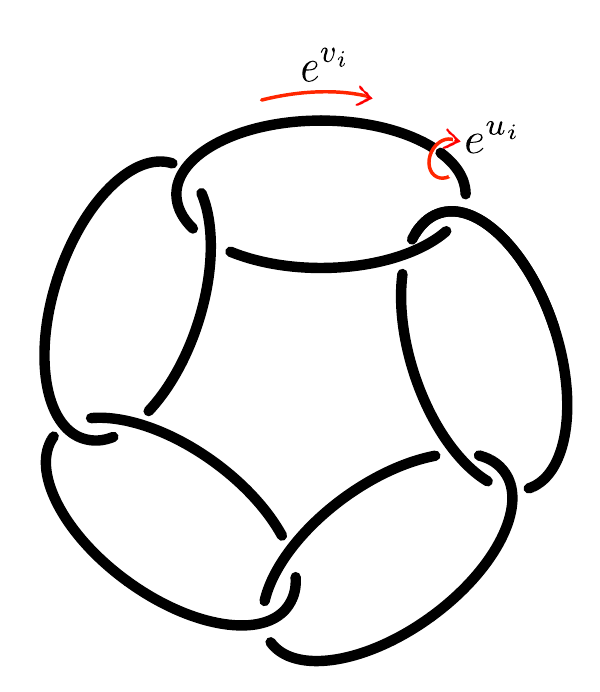}
  \vskip -.5pc
  \caption{$X$~is $S^3$ with these five solid tori removed}\label{fig:1}
  \end{figure}

  \begin{proof}
 Let $X$~be the compact spin 3-manifold with boundary formed from~$S^3$ by
removing a tubular neighborhood of the 5-component link
depicted\footnote{This link is a close cousin to the ``minimally twisted
5-chain link'' in~\cite[\S2.6]{DT}.  (We thank Ian Agol for bringing this
reference to our attention.)  } in Figure~\ref{fig:1}.  Fix a diffeomorphism
$\partial X\approx T^{\,\amalg 5}$ which induces the basis of first homology
indicated in the figure.  Alexander duality implies $H_1(X)$~is torsionfree of
rank~5, so if a principal flat $\Cx$-bundle over~$X$ admits a flat connection
then it is trivializable.  Let $\MX$ be the moduli space of flat
$\Cx$-connections on~$X$ and $r\:\MX\to(\MT)^5$ the restriction map to the
boundary~$\partial X$.  Define $\MXp=r\inv \bigl[(\MTp)^5 \bigr]$.
(Remark~\ref{thm:15} gives an explicit parametrization of~$\MXp$.)
Similarly, let $\hMX$ be the moduli space of flat $\Cx$-connections with a
homotopy class of trivialization, $\har\:\hMX\to(\hMT)^5$ restriction to the
boundary, and $\hMXp=\hr\inv \bigl[(\hMTp)^5 \bigr]$.  Each component of the
link is the outer boundary of a neatly embedded disk~$D_i\subset X$ with two
subdisks removed, as in Figure~\ref{fig:2}.  For any collection $(u_i,v_i)\in
\hMT$ in the image of~$\har$, apply Stokes' theorem to the closed 1-form
on~$D_i$ which is the flat connection form relative to the trivialization.
Its integral over a boundary component is minus the log holonomy.  Therefore,
the relation~\eqref{eq:48} holds on the image of~$\har$.

  \begin{figure}[H]
  \centering
  \includegraphics[scale=.75]{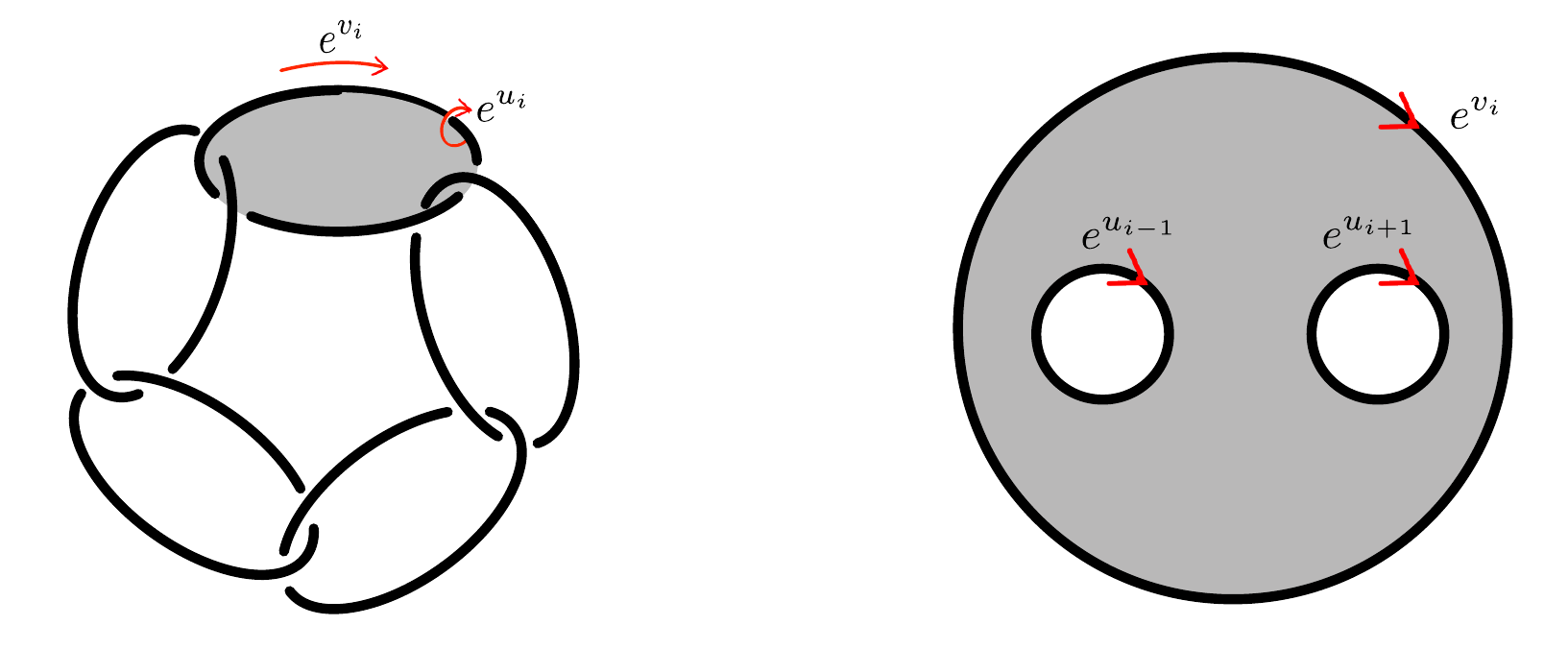}
  \vskip -.5pc
  \caption{A 2-manifold neatly embedded in~$X$}\label{fig:2}
  \end{figure}

Spin Chern-Simons theory~$\rS$ produces the diagram 
  \begin{equation}\label{eq:51}
     \begin{gathered} \xymatrix@C-2ex{&\;\;(\har') ^*(\hsL')^{\boxtimes
     5}\ar@<-.5ex>[dl] && \;\;(\hsL')^{\boxtimes 5}\ar@<-.5ex>[dl]\\ \hMXp
     \ar[dd]^{e_X'}
     \;\ar[rr]^{\har'} &&
     (\hMTp)^5\ar@[red]@<-.5ex>[ur]_<<<<<<<<<{\red(\htao)^{\boxtimes
     5}}\ar[dd]^{(e_T')^5} \\ 
     &\quad\quad(r')^*(\sL')^{\boxtimes 5}\ar@<-.5ex>[dl]&& (\sL')^{\boxtimes
     5}\ar@<-.5ex>[dl]\\ \MXp 
     \;\ar@[blue]@<-.5ex>[ur]_<<<<<<<<<{\blue\tsm} \ar[rr]^{r'} &&
     (\MTp)^5\ar@[blue]@<-.5ex>[ur]_<<<<<<<<<{\blue(\tau ')^{\boxtimes 5}}}
     \end{gathered} 
  \end{equation}
where $\hsL'=(e_T')^*\sL'$ is the pullback of the spin Chern-Simons line
bundle under $e_T'\:\hMTp\to \MTp$.  Apply Theorem~\ref{thm:s20}(iii) to
obtain a section $\ssm$ of $(r')^*(\sL')^{\boxtimes 5}\to\MXp$;
by~\eqref{eq:40} it is flat.  There exists $\tau '\in \sT'$ a flat section of
$\sL'\to\MTp$, unique up to a $5^{\textnormal{th}}$~root of unity, such that
$\ssm=(r')^*\bigl[(\tau ')^{\boxtimes 5}\bigr]$.  Then \eqref{eq:41}~implies
that the pullbacks of~$\ssm$ and~$(\htao)^{\boxtimes 5}$ to~$\hMXp$ agree.
Recalling~\eqref{eq:25}, we see that the ratio $(\htao/\tau
')^{\boxtimes5}\:(\hMTp)^5\to\Cx$ is the product of five exponentiated
enhanced Rogers dilogarithms, and the fact that its pullback to~$\hMXp$ is
identically one is the 5-term relation.
  \end{proof}

  \begin{remark}[]\label{thm:16}
 The choice of~$\tau '\in \sT'$ in the proof makes the 5-term sum in the
theorem vanish. 
  \end{remark}

\appendix

   \section{Classical spin Chern-Simons theory}\label{sec:6}
 
In this appendix we sketch proofs of Theorem~\ref{thm:s20} and
Theorem~\ref{thm:17}, which set out basic properties of spin Chern-Simons
theory with gauge group~$\Cx$.  Theorem~\ref{thm:4} can be proved by similar
methods, or may be deduced as a corollary of the spin case; we give an
argument for part~(i) at the end.  In this appendix we freely use generalized
differential cohomology, as developed in \cite{HS,BNV} and other references.
To begin we work with the universal family of $\Cx$-connections as described
in~\cite{FH}; the parameter ``space''\footnote{It is, rather, a simplicial
sheaf on the site of smooth manifolds.  Computations may be carried out on a
smooth ``test'' manifold~$M$ equipped with a $\Cx$-connection, that is,
equipped with a map $M\to \BNG$.} is~$\BNG$.

  \begin{remark}[]\label{thm:18}
 An alternative approach to spin Chern-Simons for \emph{compact} gauge groups
is available using $\eta $-invariants and pfaffian lines of Dirac
operators~\cite{J}.  For a noncompact group one encounters Dirac operators
coupled to non-unitary connections, for which the relevant parts of geometric
index theory are not in the literature.
  \end{remark}

Let $E$~be the spectrum (in the sense of stable homotopy theory) defined as
the extension
  \begin{equation}\label{eq:54}
     H\ZZ\longrightarrow E\longrightarrow \Sigma ^{-2}H\zt 
  \end{equation}
of Eilenberg-MacLane spectra with nonzero $k$-invariant.  Its properties are
stated and proved in \cite[\S1]{F3}.  The techniques used there, particularly
around~(1.13), imply that the short exact sequence
  \begin{equation}\label{eq:55}
     0\longrightarrow
     H^4(B\Cx;\ZZ)\xrightarrow{\;\;i\;\;}E^4(B\Cx)\xrightarrow{\;\;j\;\;}
     H^2(B\Cx;\zt)\longrightarrow 0 
  \end{equation}
does not split.  Hence there is a unique class $\lambda \in E^4(B\Cx)$ such
that $2\lambda =i({c_1}^2)$ and $j(\lambda )=\overline{c_1}$, where
$\overline{c_1}$~ is the mod~2 reduction of the universal first Chern class
$c_1\in H^2(B\Cx;\ZZ)$.  Let $\dEs$ be the complex differential refinement
of~$E$, defined as a homotopy fiber product as in~\cite[(4.12)]{HS} (with
$\mathcal{V}=\CC$ concentrated in degree zero); see also~\cite[\S4.4]{BNV}
(with $C=\CC$ concentrated in degree zero).  The differential cohomology
group fits into an exact sequence
  \begin{equation}\label{eq:57}
     0\longrightarrow \dE^4(\BNG)\longrightarrow E^4(B\Cx)\;\times \;\Omega
     ^4_{\textnormal{cl}}(\BNG;\CC) \xrightarrow{\;\;-\;\;} H^4(B\Cx;\CC),
  \end{equation}
part of the Mayer-Vietoris sequence derived from the homotopy fiber
product.\footnote{The short exact sequences one usually derives from it for
smooth manifolds depend on the de Rham theorem, which does not necessarily
hold for~$\BNGG$ if $G$~is noncompact.}  Here $\Omega
^4_{\textnormal{cl}}(\BNG;\CC)$ is the vector space of closed complex
differential forms; by the main theorem in~\cite{FH} it is isomorphic to the
three-dimensional complex vector space of symmetric functions $\CC\times
\CC\to\CC$ which are real bilinear.  The final map in~\eqref{eq:57} is the
difference between the homomorphism $k\:E^4(B\Cx)\to H^4(B\Cx;\CC)$ defined
in~\cite[(1.4)]{F3} and the Chern-Weil homomorphism.  It follows that the
generator $\lambda \in E^4(B\Cx)$ has a unique differential refinement
$\cla\in \dE^4(\BNG)$ with associated symmetric function
$(z_1,z_2)\longmapsto -z_1z_2/8\pi ^2$.  This is the universal class which
defines spin Chern-Simons theory.  The constructions which follow use a
geometric representative of this class, say a \emph{differential
$E$-function} as defined in \cite[\S4.1]{HS}.  Its ``curvature'', the image
under the homomorphism
  \begin{equation}\label{eq:58}
     \dE^4(\BNG)\longrightarrow \Omega ^4_{\textnormal{cl}}(\BNG;\CC) ,
  \end{equation}
is 
  \begin{equation}\label{eq:59}
     \wsp{\Tu} = -\frac{1}{8\pi ^2}\,\Omega (\Tu)\wedge \Omega (\Tu), 
  \end{equation}
where $\Tu\in \Omega ^1(\ENG;\CC)$ is the universal $\Cx$-connection
\cite[(5.25)]{FH} on the universal $\Cx$-bundle 
  \begin{equation}\label{eq:60}
     \pi \:\ENG\longrightarrow \BNG. 
  \end{equation}
Recall \cite[Example~5.14]{FH} that $\ENG$~is the classifying sheaf for
triples~$(p,\Theta ,s)$ consisting of a principal $\Cx$-bundle $p\:P\to M$
with connection~$\Theta $ and section~$s$.  The pullback of~\eqref{eq:60}
by~$\pi $ is canonically trivialized by the section, and this induces a
``nonflat trivialization''\footnote{An explicit model for a nonflat
trivialization of a geometric representative of an $\widecheck{E}$-cohomology
class---a \emph{coned differential $E$-function}---is spelled out in
\cite[Definition~5.12]{F3}.  The data of a 3-form, such as~\eqref{eq:61}, is
part of that definition.\label{foot:2}} of~$\pi ^*\cla$.  The trivialization
represents~$\pi ^*\cla$ by a complex 3-form, the universal Chern-Simons form
  \begin{equation}\label{eq:61}
     \asp{\Tu} = -\frac{1}{8\pi ^2}\, \Tu\wedge \pi^*\Omega (\Tu) 
  \end{equation}
whose de Rham differential is~\eqref{eq:59}.  Below in~\eqref{eq:68} we give
a formula for this trivialization as an integral of~$\pi ^*\cla$.

Turning now to a family $P\xrightarrow{\;\pi \;}M_\sigma
\xrightarrow{\;p\;}S$ with $\Cx$-connection~$\Theta $, as in~\eqref{eq:s59},
we obtain from these universal constructions: a differential class (rather, a
geometric representative) $\cla(\Theta )\in \dE^4(M)$ with
curvature~$\wsp{\Theta }$, as in~\eqref{eq:59}; the Chern-Simons form
$\asp\Theta \in \Omega ^3(P;\CC)$; and an isomorphism of~$\pi ^*\cla(\Theta
)$ with the image of~$\asp{\Theta }$ in~$\dE^4(P)$.  The spin Chern-Simons
invariant is the pushforward in differential $E$-theory: 
  \begin{equation}\label{eq:62}
     \thS{M_\sigma /S}\Theta =\int_{M_\sigma /S}\cla(\Theta ). 
  \end{equation}
Integration in~$E$, and so in~$\dEs$, uses the spin structure~$\sigma $;
see~\cite[(1.6)]{F3}.  If the fibers of~$p$ are closed of dimension~$k$, then
\eqref{eq:62} lives in~$\dE^{4-k}(S)$.  In low dimensions we identify
  \begin{equation}\label{eq:63}
     \begin{aligned} \dE^1(S)&\cong \Map(S,\Cx) \\ \dE^2(S)&\cong \bigl\{
      \textnormal{iso classes of complex super line bundles $\sL\to
      S$ with connection} \bigr\} \end{aligned} 
  \end{equation}
 
Parts (i), (v), and (vi) of Theorem~\ref{thm:s20} follow immediately from the
foregoing and the compatibility of the exact sequence~\eqref{eq:57} and
integration.  The $\zt$-grading in~(v) is determined by its restriction to
each~$s\in S$, and since the grading is a discrete invariant it only depends
on $\lambda (\Theta )\in E^4(M)$.  Writing $M_s=p\inv (s)$, the desired
formula follows from the commutative diagram 
  \begin{equation}\label{eq:81}
     \begin{gathered} \xymatrix{E^4(M_s)\ar[r]^<<<<<{j} \ar[d]_{p_*} &
     H^2(M_s;\zt)\ar[d]^{p_*} \\ E^2\bigl(\{s\} \bigr)\ar[r]^<<<<<{j} &
     H^0\bigl(\{s\};\zt \bigr) } \end{gathered} 
  \end{equation}
in which both maps~$j$ are isomorphisms and $j\bigl(\lambda (\Theta )
\bigr)=\overline{c_1}(P)$.  The dependence on spin structure in~(ii) is an
immediate consequence of the proof of \cite[Proposition~4.4]{F3}.  For
parts~(iii) and~(iv) use the notion of a nonflat trivialization of a
geometric representative of an $\dEs$-cohomology class; see
footnote~\footref{foot:2}.  If the fibers of $p\:M\to S$ are compact
manifolds with boundary, a Stokes' theorem holds (see \cite[\S3.4]{HS} for
one version):
  \begin{equation}\label{eq:64}
  \begin{gathered}
      \textnormal{$\int_{M_\sigma /S}\cla(\Theta )$ is a nonflat
     trivialization of $\int_{\partial M_\sigma /S}\cla(\Theta )$}\\\textnormal{ with
     ``covariant derivative'' $\int_{M/S}\wsp{\Theta }$.} 
  \end{gathered}
  \end{equation}
For $\dim(p)=3$ the nonflat trivialization is a nonzero section of a
(necessarily even) complex line bundle and (iii)~follows.  For~(iv) we use
the isomorphism of~$\cla(\Theta )$ with the image of~$s^*\asp\Theta $
in~$\dE^4(M)$ and apply ~\eqref{eq:64}.

We turn to Theorem~\ref{thm:s20}(vii), and for that we prove a formula valid
for any finite dimensional real Lie group~$G$ with finitely many components.
Analogous to~\eqref{eq:57} is the exact sequence 
  \begin{equation}\label{eq:65}
     0\longrightarrow \dE^4(\BNGG)\longrightarrow E^4(BG)\;\oplus \;\Omega
     ^4_{\textnormal{cl}}(\BNGG) \xrightarrow{\;\;-\;\;} H^4(BG;\CC), 
  \end{equation}
and the main theorem in~\cite{FH} identifies~$\Omega
^4_{\textnormal{cl}}(\BNGG)$ with the complex vector space of $G$-invariant
symmetric bilinear forms $\mathfrak{g}\times \mathfrak{g}\to\CC$ on the Lie
algebra of~$G$.  It follows that $\dE^4(\BNGG)$ is isomorphic to the group of
compatible pairs $\cla=\bigl(\lambda ,\langle -,- \rangle \bigr)$ of a class
in~$E^4(BG)$ and a symmetric bilinear form.  Fix such a pair.  Let $\pi
\:\ENGG\to\BNGG$ be the universal $G$-bundle with universal $G$-connection
$\Tu\in \Omega ^1(\ENGG;\mathfrak{g})$.  The total space~$\sF$ of the fiber
product of~$\pi $ with itself classifies quartets $(p,\Theta ,s_0,s_1)$
consisting of a principal $G$-bundle $P\to M$ over a smooth manifold~$M$
equipped with connection~$\Theta $ and sections~$s_0,s_1$.  Use~$s_0$ to
trivialize the (universal) bundle, and so construct an isomorphism
  \begin{equation}\label{eq:66}
     \sF\xrightarrow{\;\;\cong \;\;}\ENGG\times G. 
  \end{equation}
The sections~$s_0,s_1$ each induce a trivialization of the pullback of~$\cla$
to~$\sF$; the difference of these trivializations represents an element
$\ce\in \dE^3(\sF)$.

  \begin{proposition}[]\label{thm:19}
 Under the isomorphism~\eqref{eq:66} the differential class~$\ce$ is the sum 
  \begin{equation}\label{eq:67}
     \ce=\co \;+\;\langle \Tu\wedge \theta \rangle, 
  \end{equation}
where $\co\in \dE^3(G)$ is the integral of~$\cla(\tT)$ over~$\cir$ for a
certain $G$-connection~$\tT$ over~$\cir\times G$, and $\theta \in \Omega
^1_G(\mathfrak{g})$ is the Maurer-Cartan form.
  \end{proposition}

\noindent 
 $\tT$ is a universal pointed connection: its holonomy around $\cir\times
\{h\}$ equals\footnote{The signs work out so that the holonomy
is~$h\inv $, not~$h$.}~$h\inv $ for all~$h\in G$.

  \begin{proof}
 The pullback of the universal $G$-bundle $\pi \:\ENGG\to\BNGG$ via~$\pi $ is
a principal $G$-bundle $\varpi \:\sF\to\ENGG$ equipped with a canonical
section~$s$.  The latter induces the canonical trivialization~\eqref{eq:66}
as well as a canonical trivial connection~$\Theta _s$, which
under~\eqref{eq:66} maps to the Maurer-Cartan form~$\theta $.  Let $\Delta
^1\to \conn\varpi $ be the affine map of the 1-simplex into the affine space
of connections on~$\varpi $ which sends the endpoints to~$\Theta _s$ and~$\pi
^*\Tu$.  There results a connection~$\tT_s$ on the trivial $G$-bundle over
$\Delta ^1\times \ENGG$.  By Stokes' theorem~\eqref{eq:64} the nonflat
trivialization of~$\pi ^*\cla(\Tu)$ is
  \begin{equation}\label{eq:68}
     \int_{\Delta ^1}\cla(\tT_s) = \int_{\Theta _s}^{\pi ^*\Tu}\cla.
  \end{equation}
(The right hand side is shorthand notation for the left hand side.)  The
universal Chern-Simons form~\eqref{eq:61} is the same integral with the
integrand replaced by its curvature $\langle \Omega (\tT_s)\wedge \Omega
(\tT_s) \rangle$.

 \begin{figure}[h]
  \centering
  \includegraphics[scale=.6]{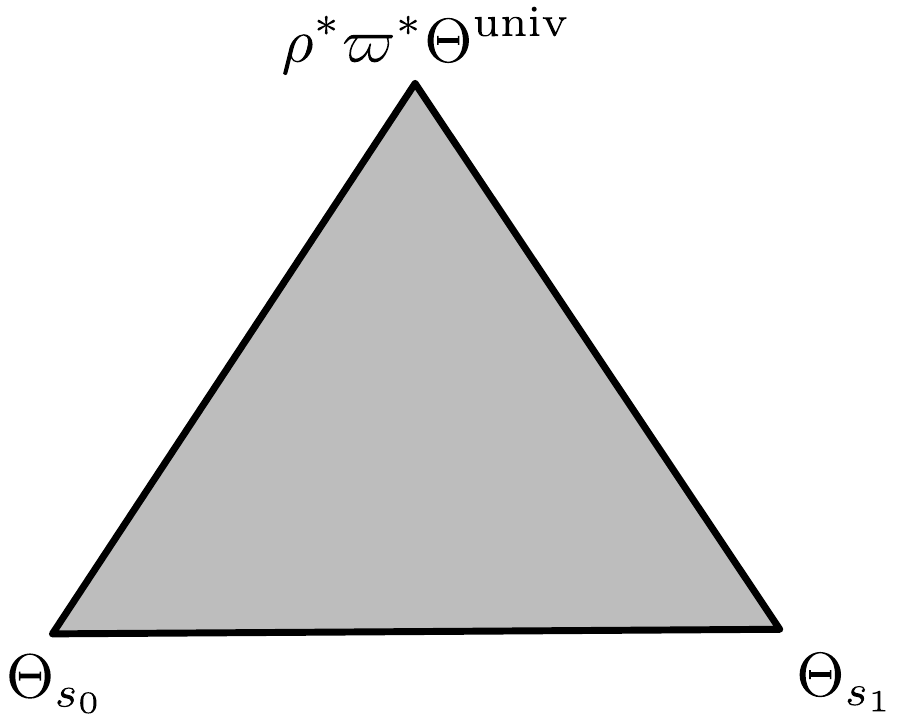}
  \vskip -.5pc
  \caption{The 2-simplex of connections $\Delta ^2\to\conn\rho $}\label{fig:3}
  \end{figure}
 
Working now on the universal bundle $\rho \:\sP\to\sF$, which is equipped
with two sections~$s_0$ and $s_1$, there is an affine map $\Delta
^2\to\conn\rho $ as depicted in Figure~\ref{fig:3}.  The class~$\ce$ is
obtained by integrating~$\cla$ over the path which begins at the lower left
vertex, moves to the top vertex, and then down to the lower right vertex; we
use a gauge transformation to identify the trivial connections at the initial
and final vertices.  A version of Stokes' theorem identifies the difference
of the trivializations of the pullback of~$\cla$ to~$\sF$ induced by~$s_1$
and~$s_0$ as
  \begin{equation}\label{eq:69}
     \ce = \int_{\Theta _{s_0}}^{\Theta _{s_1}}\cla \;-\; \int_{\Delta
     ^2}\langle \Omega (\tT)\wedge \Omega (\tT) \rangle, 
  \end{equation}
where $\tT$~is the $G$-connection over~$\Delta ^2\times \sF$ constructed by
affine interpolation from the vertices in Figure~\ref{fig:3}.  As before, we
identify the connections~$\Theta _{s_0}$ and~$\Theta _{s_1}$ by a gauge
transformation.   In terms of barycentric coordinates on~$\Delta ^2$, the
map $\Delta ^2\to\conn\rho $ is
  \begin{equation}\label{eq:70}
     (t^0,t^1,t^2)\longmapsto t^0\Theta _{s_0} + t^1\Theta _{s_1} + t^2\rho
     ^*\varpi ^*\Tu. 
  \end{equation}
Write $s_1=s_0\cdot h$ for $h\:\sF\to G$.  Relative to the
trivialization~$s_0$ and isomorphism~\eqref{eq:66} we rewrite\footnote{Since
$s_1^*\Theta _{s_1}=0$, the usual formula for the gauge transform of a
connection implies $0=\theta +\Ad_{h\inv }(s_0^*\Theta \mstrut
_{s_1})$.}~\eqref{eq:70} as a map to $\Omega ^1_{\ENGG\times
G}(\mathfrak{g})$:
  \begin{equation}\label{eq:71}
     (t^0,t^1,t^2)\longmapsto -t^1\Ad_h\theta \,+\, t^2\Tu. 
  \end{equation}
The last expression, interpreted as an element of~$\Omega ^1_{\Delta ^2\times
\ENGG\times G}(\mathfrak{g})$, is precisely~$\tT$.  Its curvature is 
  \begin{equation}\label{eq:72}
     \Omega (\tT) = -dt^1\wedge \Ad_h\theta \;+\; dt^2\wedge \Tu \;+\;
     \textnormal{terms not involving $dt^1$ or $dt^2$}, 
  \end{equation}
from which 
  \begin{equation}\label{eq:73}
     -\int_{\Delta ^2}\langle \Omega (\tT)\wedge \Omega (\tT) \rangle =
     \langle \Tu\wedge \theta \rangle. 
  \end{equation}

To identify the first term in~\eqref{eq:69} in terms of a $G$-connection
over~$\cir\times G$, restrict~\eqref{eq:71} to the 1-simplex $t^2=0$ to
obtain the connection form $-t^1\Ad_h\theta $.  Compare
with~\cite[(4.14)]{F1}.\footnote{Set $g=1$ in~\cite[(4.14)]{F1} to compare;
the sign discrepancy is explained by the appearance of the inverse `$h\inv $'
in the gluing formula~\cite[(4.13)]{F1}.}
  \end{proof}

Apply Proposition~\ref{thm:19} with $G=\Cx$ to Theorem~\ref{thm:s20}(vii).
First, the universal $\Cx$-bundle $Q\to\scx$ has first Chern class a
generator of $H^2(\scx;\ZZ)$ (as follows from~\eqref{eq:22}, for example), so
from~\eqref{eq:55} and~\eqref{eq:57}---applied to~$\scx$---we see that
$\cla(Q)$~is the generator of $\dE^4(\scx)\cong H^2(\scx;\zt)\cong \zt$.  Its
integral over~$\cir$ is the nonzero element
  \begin{equation}\label{eq:74}
     \co\in \dE^3(\Cx)\cong H^1(\Cx;\zt)\cong \zt. 
  \end{equation}

Turning to Theorem~\ref{thm:s20}(vii), the integrand in~\eqref{eq:s64} equals
the second term of~\eqref{eq:67}, so it remains to identify the
sign~\eqref{eq:s65} with the integral of the first term in~\eqref{eq:67}.
Since the curvature of~$\co$ vanishes, it suffices to compute the pushforward
of its image $\omega \in E^3(\Cx)$.  Since $E$~is a truncation of connective
$ko$-theory, there is an isomorphism 
  \begin{equation}\label{eq:75}
     E^3(\Cx) \cong ko^{-1}(\Cx)\cong ko^{-2}(\Cx;\RZ)\cong \zt .
  \end{equation}
Set $\beta \in ko^{-2}(\Cx;\RZ)$ the image of~$\omega $ under~\eqref{eq:75}.
 
Let $M$~be a closed 2-manifold with spin structure~$\sigma $ and a map
$h\:M\to\Cx$.  Our task is to equate~\eqref{eq:s65} with $\pi ^M_*h^*\beta
\in ko^{-4}(\pt;\RZ)\cong \RZ$, where $\pi ^M\:M\to\pt$.  First, deformation
retract~$\Cx$ to the circle group~$\TT\subset \Cx$ and so homotope~$h$ to a
map with image~$\TT$.  Then the generator ~$\beta \in ko^{-2}(\TT;\RZ)$ is
pushed forward from the generator $\alpha \in ko^{-3}(\pt;\RZ)\cong \zt$ via
the inclusion $e\:\pt\hookrightarrow \TT$.  By a further homotopy make $h$~
transverse to~$e\in \TT$; then its inverse image $S\subset M$ is a finite
union of disjoint embedded circles, and $S$~ inherits a spin structure
from~$M$.  Arguing from the diagram
  \begin{equation}\label{eq:d100}
     \begin{gathered} \xymatrix@C+4ex@R+4ex{S\ar[r]^{i} \ar[d]_{q} &
     M\ar[d]^{h}\ar[dl]|{{\pi ^M}_{\ }^{\ }} \\ 
     \pt\ar[r]^{e} & \TT} \end{gathered} 
  \end{equation}
we have
  \begin{equation}\label{eq:d101}
 \pi ^M_*h^*\beta =\pi ^M_*h^*e_*\alpha =\pi ^M_*i_*q^*\alpha  = q_*a^*\alpha
=\alpha \cdot q_*(1). 
  \end{equation}
Restricted to a component of~$S$, the pushforward $q_*(1)\in
KO^{-1}(\pt)\cong \zt$ is~0 or~1 according as the spin structure on the
component bounds or not.  Since the homology class of~$S$ is Poincar\'e dual
to~$[h]\in H^1(M;\ZZ)$, we conclude that $q_*(1)$ maps to $\sigma \bigl([h]
\bigr)$ under the isomorphism $KO\inv (\pt)\cong \zt$.  It remains to observe
that multiplication induces a nonzero pairing
  \begin{equation}\label{eq:d102}
     KO^{-3}(\pt;\RZ)\otimes KO^{-1}(\pt)\longrightarrow KO^{-4}(\pt;\RZ) ,
  \end{equation}
as proved for example in~\cite[(B.10)]{FMS}.

The functoriality properties in Theorem~\ref{thm:17} follow
from~\eqref{eq:62} and the functoriality of the differential characteristic
class $\cla\in \dE^4(\BNG)$.

Finally, we sketch a proof of~\eqref{eq:12}.  Let $\cco\in \dH^2(\BNG)$ be
the universal differential first Chern class for principal $\Cx$-bundles; it
is the differential lift of~$c_1\in H^2(B\Cx;\ZZ)$ with associated linear
function $z\mapsto \sqmo z/2\pi $, $z\in \CC$.  Its square $\cco\cdot \cco\in
\dH^4(\BNG)$ is the universal Chern-Simons class.  Let $\Theta $~be a flat
connection on a principal $\Cx$-bundle $\pi \:P\to X$, where $X$~is a closed
oriented 3-manifold.  Then since $\cco(\Theta )$~is flat, the product
$\cco(\Theta )\cdot \cco(\Theta )$ is computed by a cup product in cohomology
with $\CZ$~coefficients, precisely as in~\eqref{eq:12}.

   \section{A motivating Euler-Lagrange equation}\label{sec:7}
 
In our construction (\S\ref{subsec:4.1}) of the dilogarithm, we
introduce~\eqref{eq:n15} the submanifold $\MTp\subset \MT$ of flat
$\Cx$-connections on the torus $T=\RR^2/\ZZ^2$ such that the holonomies
around the standard cycles sum to one.  This condition arises naturally in
the stratified abelianization of flat $\SL_2\CC$-connections~\cite{FN}.  In
this appendix we briefly indicate a formal computation motivated by the
topological string~\cite{W}, \cite{OV} that produces this condition
on holonomies.
 
Let $M$~be a closed spin 3-manifold and $S\subset M$ an oriented embedded
circle.  For $\alpha \in \Omega ^1_M(\CC)$ a connection on the trivial
$\Cx$-bundle over~$M$, introduce 
  \begin{equation}\label{eq:77}
     F(\alpha ) = -\frac{1}{8\pi ^2}\int_{M}\alpha \wedge d\alpha
     \;-\;\frac{1}{4\pi ^2}\Li_2\left[ \exp\left( -\int_{S}\alpha \right)
     \right] . 
  \end{equation}
In this expression $\Li_2$~is the Spence dilogarithm~\eqref{eq:31} evaluated
at the holonomy of~$\alpha $ about~$S$.  Since $\Li_2$~is not a global
function on~$\CC$---see~\eqref{eq:33}---this is ill-defined, so our
computation based on~\eqref{eq:77} is heuristic.    The first term
in~\eqref{eq:77} is the spin Chern-Simons invariant (see~\eqref{eq:61}), and
the normalization of the second term matches that of the first: we should
view~$F$ as defined modulo integers.  The differential of~$F$ is
  \begin{equation}\label{eq:78}
     dF_\alpha (\da) = -\frac{1}{4\pi ^2}\int_{M}\da\wedge \bigl[ d\alpha +
     \log(1-z)\delta _S\bigr] ,\qquad \da\in \Omega ^1_M(\CC), 
  \end{equation}
where $z=\exp(-\int_{S}\alpha )$ is the holonomy of~$\alpha $ about~$S$ and
$\delta _S$~is the distributional 2-form Poincar\'e dual to~$S$.  The
critical point (Euler-Lagrange) equation is
  \begin{equation}\label{eq:79}
     d\alpha =-\log(1-z)\delta _S. 
  \end{equation}
A critical~$\alpha $ is flat on the complement of~$S$ and the holonomy around
a small loop linking~$S$ is $\exp\bigl\{-\bigl[-\log(1-z) \bigr] \bigr\}=1-z$.
Therefore, on the torus boundary of a tubular neighborhood of~$S$, a
critical connection~$\alpha $ is flat and the sum of holonomies about
generating cycles is one. 

  \begin{remark}[]\label{thm:20} 
 We can define~$F$ on a cover of the space of connections whose holonomy~$z$
about~$S$ is not equal to one, namely the cover on which we choose logarithms
for~$z$ and~$1-z$.  But the critical point equation~\eqref{eq:79} takes us to
a space of singular connections.  In particular, the holonomy about~$S$ is no
longer defined.  To rectify this, we can from the beginning choose an
embedding of~$[0,1)\times \cir\hookrightarrow M$ such that the image
of~$\{0\}\times \cir$ is~$S$, and then replace the argument of~$\Li_2$
in~\eqref{eq:77} with the limit of holonomies around loops $\{t\}\times \cir$
as~$t\to0$.  This produces a basis of the first homology of the torus
boundary of a tubular neighborhood of~$S$, so pins down the cycles on which
the holonomies sum to one. 
  \end{remark}

  \begin{remark}[]\label{thm:21}  
 One possible origin of the action~\eqref{eq:77} is as follows.  We suppose
$M$ is a Lagrangian submanifold inside a Calabi-Yau threefold $X$, and we
view $(M, \alpha)$ as defining a Lagrangian boundary condition in the A type
topological string theory of maps $\varphi: (\Sigma, \partial \Sigma) \to
(X,M)$.  Then $F(\alpha)$ ~can be interpreted as the target space effective
action of that string theory. This prediction comes from combining~\cite{W}
and~\cite{OV}: in \cite[(4.50)]{W} it was explained that the effective action
should be the Chern-Simons action plus additional contributions from
holomorphic maps (if any), and in \cite[(3.22)]{OV} it was shown that the
contribution from an isolated holomorphic disc $\varphi$ is $\Li_2(e^{-x}) /
4 \pi^2$, where $x = \int_{\partial D} \phi^* \alpha$.
  \end{remark}

\newpage
  \bigskip\bigskip

\providecommand{\bysame}{\leavevmode\hbox to3em{\hrulefill}\thinspace}
\providecommand{\MR}{\relax\ifhmode\unskip\space\fi MR }
\providecommand{\MRhref}[2]{%
  \href{http://www.ams.org/mathscinet-getitem?mr=#1}{#2}
}
\providecommand{\href}[2]{#2}


\begin{thebibliography}{BNV16}

\bibitem[APP]{APP}
Sergei Alexandrov, Daniel Persson, and Boris Pioline, \emph{Wall-crossing,
  Rogers dilogarithm, and the QK/HK correspondence}, Journal of High Energy
  Physics \textbf{2011} (2011), no.~12, 27.

\bibitem[BNV]{BNV}
Ulrich Bunke, Thomas Nikolaus, and Michael V\"{o}lkl, \emph{Differential
  cohomology theories as sheaves of spectra},
  \href{http://dx.doi.org/10.1007/s40062-014-0092-5}{J. Homotopy Relat. Struct.
  \textbf{11} (2016)}, no.~1, 1--66.

\bibitem[ChS]{ChS}
Jeff Cheeger and James Simons,
  \href{http://dx.doi.org/10.1007/BFb0075216}{\emph{Differential characters and
  geometric invariants}}, Geometry and topology ({C}ollege {P}ark, {M}d.,
  1983/84), Lecture Notes in Math., vol. 1167, Springer, Berlin, 1985,
  pp.~50--80.

\bibitem[CS]{CS}
Shiing~Shen Chern and James Simons, \emph{Characteristic forms and geometric
  invariants}, Ann. of Math. (2) \textbf{99} (1974), 48--69.

\bibitem[D]{D}
Johan~L. Dupont,
  \href{http://dx.doi.org/10.1016/0022-4049(87)90021-1}{\emph{The dilogarithm
  as a characteristic class for flat bundles}}, Proceedings of the
  {N}orthwestern conference on cohomology of groups ({E}vanston, {I}ll., 1985),
  vol.~44, 1987, pp.~137--164.

\bibitem[DS]{DS}
Johan~L. Dupont and Chih~Han Sah, \emph{Scissors congruences. {II}},
  \href{http://dx.doi.org/10.1016/0022-4049(82)90035-4}{J. Pure Appl. Algebra
  \textbf{25} (1982)}, no.~2, 159--195.

\bibitem[DT]{DT}
Nathan~M. Dunfield and William~P. Thurston, \emph{The virtual {H}aken
  conjecture: experiments and examples},
  \href{http://dx.doi.org/10.2140/gt.2003.7.399}{Geom. Topol. \textbf{7}
  (2003)}, 399--441, \href{http://arxiv.org/abs/arXiv:math/0209214}{{\tt
  arXiv:math/0209214}}.

\bibitem[F1]{F1}
Daniel~S. Freed, \emph{Classical {C}hern-{S}imons theory. {I}}, Adv. Math.
  \textbf{113} (1995), no.~2, 237--303,
  \href{http://arxiv.org/abs/arXiv:hep-th/9206021}{{\tt arXiv:hep-th/9206021}}.

\bibitem[F2]{F2}
\bysame, \emph{Classical {C}hern-{S}imons theory. {II}}, Houston J. Math.
  \textbf{28} (2002), no.~2, 293--310. Special issue for S. S. Chern.

\bibitem[F3]{F3}
Daniel~S. Freed, \emph{Pions and generalized cohomology}, J. Differential Geom.
  \textbf{80} (2008), no.~1, 45--77,
  \href{http://arxiv.org/abs/arXiv:hep-th/0607134}{{\tt arXiv:hep-th/0607134}}.

\bibitem[FH]{FH}
Daniel~S. Freed and Michael~J. Hopkins, \emph{Chern-{W}eil forms and abstract
  homotopy theory},
  \href{http://dx.doi.org/10.1090/S0273-0979-2013-01415-0}{Bull. Amer. Math.
  Soc. (N.S.) \textbf{50} (2013)}, no.~3, 431--468,
  \href{http://arxiv.org/abs/arXiv:1301.5959}{{\tt arXiv:1301.5959}}.

\bibitem[FMS]{FMS}
Daniel~S. Freed, Gregory~W. Moore, and Graeme Segal, \emph{The uncertainty of
  fluxes}, \href{http://dx.doi.org/10.1007/s00220-006-0181-3}{Commun. Math.
  Phys. \textbf{271} (2007)}, 247--274,
\href{http://arxiv.org/abs/hep-th/0605198}{{\tt arXiv:hep-th/0605198}}.

\bibitem[FN]{FN}
Daniel~S. Freed and Andrew Neitzke, \emph{3d spectral networks}. in
  preparation.

\bibitem[G]{G}
Alexander~B. Goncharov, \emph{Polylogarithms in arithmetic and geometry},
  Proceedings of the {I}nternational {C}ongress of {M}athematicians, {V}ol. 1,
  2 ({Z}\"{u}rich, 1994), Birkh\"{a}user, Basel, 1995, pp.~374--387.

\bibitem[HS]{HS}
M.~J. Hopkins and I.~M. Singer, \emph{Quadratic functions in geometry,
  topology, and {M}-theory}, J. Diff. Geom. \textbf{70} (2005), 329--452,
\href{http://arxiv.org/abs/math/0211216}{{\tt arXiv:math/0211216}}.

\bibitem[J]{J}
J.~A. Jenquin, \emph{Spin Chern-Simons and spin TQFTs},
  \href{http://arxiv.org/abs/arXiv:math/0605239}{{\tt arXiv:math/0605239}}.

\bibitem[K]{K}
Anatol~N. Kirillov,
  \href{http://dx.doi.org/10.1143/PTPS.118.61}{\emph{Dilogarithm identities}},
  Quantum field theory, integrable models and beyond (Kyoto, 1994), no. 118,
  1995, pp.~61--142. \href{http://arxiv.org/abs/arXiv:hep-th/9408113}{{\tt
  arXiv:hep-th/9408113}}.

\bibitem[LL]{LL}
Jean Lannes and Francois Latour, \emph{forme quadratique d'enlancement et
  applications}, ast\'erique, vol.~26, soci\'ete math\'ematique de france,
  1975.

\bibitem[MS]{MS}
John~W Morgan and Dennis~P Sullivan, \emph{The transversality characteristic
  class and linking cycles in surgery theory}, Annals of Mathematics (1974),
  463--544.

\bibitem[OV]{OV}
Hirosi Ooguri and Cumrun Vafa, \emph{Knot invariants and topological strings},
  \href{http://dx.doi.org/10.1016/S0550-3213(00)00118-8}{Nuclear Phys. B
  \textbf{577} (2000)}, no.~3, 419--438,
  \href{http://arxiv.org/abs/arXiv:hep-th/9912123}{{\tt arXiv:hep-th/9912123}}.

\bibitem[S]{S}
William Spence, \emph{An essay on the theory of the various orders of
  logarithmic transcendents}, John Murray and Archibald Constable and Company,
  1809.

\bibitem[Se]{Se}
Graeme Segal, \emph{Felix Klein Lectures 2011}.
  \url{http://www.mpim-bonn.mpg.de/node/3372/abstracts}.

\bibitem[ST]{ST}
Stephan Stolz and Peter Teichner,
  \href{http://dx.doi.org/10.1090/pspum/083/2742432}{\emph{Supersymmetric field
  theories and generalized cohomology}}, Mathematical foundations of quantum
  field theory and perturbative string theory, Proc. Sympos. Pure Math.,
  vol.~83, Amer. Math. Soc., Providence, RI, 2011, pp.~279--340.
  \href{http://arxiv.org/abs/arXiv:1108.0189}{{\tt arXiv:1108.0189}}.

\bibitem[W]{W}
E.~Witten, \emph{Chern-{S}imons gauge theory as a string theory}, The {F}loer
  memorial volume, Progr. Math., vol. 133, Birkh\"{a}user, Basel, 1995,
  pp.~637--678. \href{http://arxiv.org/abs/arXiv:hep-th/9207094}{{\tt
  arXiv:hep-th/9207094}}.

\bibitem[Z]{Z}
Don Zagier, \href{http://dx.doi.org/10.1007/978-3-540-30308-4_1}{\emph{The
  dilogarithm function}}, Frontiers in number theory, physics, and geometry.
  {II}, Springer, Berlin, 2007, pp.~3--65.

\end{thebibliography}
  \end{document}